\documentclass[reqno]{amsart}
\usepackage{CJK}
\usepackage{hyperref}
\usepackage{cite}
\usepackage{amsmath}
\usepackage{mathrsfs}
\def\dive{\operatorname{div}}

\numberwithin{equation}{section}

\newtheorem{theorem}{Theorem}[section]
\newtheorem{lemma}[theorem]{Lemma}
\newtheorem{definition}[theorem]{Definition}

\newtheorem{remark}[theorem]{Remark}
\newtheorem{corollary}[theorem]{Corollary}

\allowdisplaybreaks

\newcommand{ \mint }{ {\int\hspace{-0.38cm}-}}

\begin{document}
	
\title[\hfil Regularity theory for mixed local and nonlocal\dots] {Regularity theory for mixed local and nonlocal parabolic $p$-Laplace equations}

\author[B. Shang, Y. Fang and C. Zhang  \hfil \hfilneg]
{Bin Shang, Yuzhou Fang and Chao Zhang$^*$}

\thanks{$^*$Corresponding author.}

\address{Bin Shang \hfill\break
School of Mathematics, Harbin Institute of Technology,
Harbin 150001, P.R. China} \email{shangbin0521@163.com}

\address{Yuzhou Fang \hfill\break
School of Mathematics, Harbin Institute of Technology,
Harbin 150001, P.R. China} \email{18b912036@hit.edu.cn}

\address{Chao Zhang\hfill\break
School of Mathematics and Institute for Advanced Study in Mathematics, Harbin Institute of Technology,
Harbin 150001, P.R. China} \email{czhangmath@hit.edu.cn}

.

\maketitle

\begin{abstract}
	We investigate the mixed local and nonlocal parabolic $p$-Laplace equation
	\begin{align*}
	\partial_t u(x,t)-\Delta_p u(x,t)+\mathcal{L}u(x,t)=0,
	\end{align*}
	where $\Delta_p$ is the local $p$-Laplace operator and
	$\mathcal{L}$ is the nonlocal $p$-Laplace operator. Based on the combination of suitable Caccioppoli-type inequality and Logarithmic Lemma with a De Giorgi-Nash-Moser iteration, we establish the local boundedness and H\"{o}lder continuity of weak solutions for such equations.
\end{abstract}

\section{Introduction}
\label{sec1}
\par

In this paper, we are concerned with the local behaviour of weak solutions to the following mixed problem
\begin{align}
\label{1.1}
\partial_t u(x,t)-\Delta_p u(x,t)+\mathcal{L}u(x,t)=0 \quad \text{in } Q_T,\quad  1<p<\infty,
\end{align}
where $Q_T:=\Omega\times(0,T)$ with $T>0$ and $\Omega$ is a bounded domain in $\mathbb{R}^N$. This kind of evolution equations arises from the L\'{e}vy process, image processing etc; see \cite{DV21} and references therein.
The local $p$-Laplace operator $\Delta_p$ is defined as
\begin{align*}
\Delta_p u:=\dive(|\nabla u|^{p-2}\nabla u),
\end{align*}
and $\mathcal{L}$ is a nonlocal $p$-Laplace operator given by
\begin{align}
\label{1.3}
\mathcal{L}u(x,t)=\mathrm{P.V.}\int_{\mathbb{R}^N}K(x,y,t)|u(x,t)-u(y,t)|^{p-2}(u(x,t)-u(y,t))\,dy,
\end{align}
where the symbol $\mathrm{P.V.}$ stands for the Cauchy principle value. Here $K$ is a symmetric kernel fulfilling
$$
K(x,y,t)=K(y,x,t)
$$
and
\begin{align}
\label{1.4}
\frac{\Lambda^{-1}}{|x-y|^{N+sp}}\leq K(x,y,t)\leq \frac{\Lambda}{|x-y|^{N+sp}}
\end{align}
with $\Lambda\geq1$ and $0<s<1$ for all $x,y\in\mathbb{R}^N$ and $t\in(0,T)$.

Before stating our main results, let us mention some known results. For the nonlocal parabolic equations of $p$-Laplacian type,
\begin{equation}
\label{1.4.1}
\partial_tu(x,t)+\mathcal{L}u(x,t)=0,
\end{equation}
the existence and uniqueness of strong solutions were verified by V\'{a}zquez \cite{Vaz16}, where the author studied the long time behaviours as well. Maz\'{o}n-Rossi-Toledo \cite{MRT16} established the well-posedness  of solutions to Eq. \eqref{1.4.1} together with the asymptotic property. When it comes to regularity theory for this equation, Str\"{o}mqvist \cite{Str19} obtained the existence and local boundedness of weak solutions provided $p\ge 2$.  H\"{o}lder regularity with specific exponents in the case $p\ge 2$ was proved by  Brasco-Lindgren-Str\"{o}mqvist  \cite{BLS}.  Furthermore, Ding-Zhang-Zhou \cite{DZZ21} showed the local boundedness and H\"{o}lder continuity of weak solutions to the nonhomogeneous case under the conditions that $1<p<\infty$ and $2<p<\infty$ respectively. We refer the readers to \cite{CCV11,FK13,Kim19,Vaz20,Vaz} and references therein for more results.

In the mixed local and nonlocal setting, for the case $p=2$,
\begin{equation}
\label{1.4.2}
-\Delta u+(-\Delta)^su=0,
\end{equation}
Foondun \cite{Fo9} have derived Harnack inequality and interior H\"{o}lder estimates for nonnegative solutions, see also \cite{CKSV12} for a diverse approach. In addition, the Harnack inequality regarding the parabolic version of \eqref{1.4.2} was established in \cite{BBCK9,CK10}, where however the authors only proved such inequality for globally nonnegative solutions. Very recently, Garain-Kinnunen \cite{GK21} proved a weak Harnack inequality with a tail term for sign changing solutions to the parabolic problem of \eqref{1.4.2}. For what concerns maximum principles, interior sobolev regularity along with symmetry results among many other quantitative and qualitative properties for solutions to \eqref{1.4.2}, one can see for instance \cite{BDVV20,BDVV,DLV,DLV21,DRSV}. In the nonlinear framework (i.e., $p\neq2$), Garain-Kinnunen \cite{GK} developed the local regularity theory for
$$
-\Delta_pu+\mathrm{P.V.}\int_{\mathbb{R}^N}K(x,y)|u(x)-u(y)|^{p-2}(u(x)-u(y))\,dy=0
$$
with $K(x,y)\simeq |x-y|^{-(N+sp)}$, involving boundedness, H\"{o}lder continuity, Harnack inequality as well as lower/upper semicontinuity of weak supersolutions/subsolutions.
Nonetheless, to the best of our knowledge, there are little results concerning on the mixed local and nonlocal nonlinear parabolic problems. To this end, we aim to establish the interior H\"{o}lder regularity and boundedness of weak slutions to Eq. \eqref{1.1}, which could be regarded as an analogue of \cite{DZZ21,GK}. It is noteworthy that our results are new even for the case $p=2$.

Before giving the notion of weak solutions to \eqref{1.1}, let us recall the tail space
\begin{align*}
L_\alpha^q(\mathbb{R}^{N}):=\left\{v \in L_{\rm{loc}}^q(\mathbb{R}^{N}): \int_{\mathbb{R}^N} \frac{|v(x)|^q}{1+|x|^{N+\alpha}}\,dx<+\infty\right\}, \quad q>0 \text{ and } \alpha>0.
\end{align*}
Then we define the tail appearing in estimates throughout this article,
\begin{align}
\label{1.5}
\mathrm{Tail}_\infty(v;x_{0},r,I) &=\mathrm{Tail}_{\infty}(v; x_0,r,t_0-T_1,t_0+T_2)\nonumber\\
&:=\operatorname*{ess\,\sup}_{t \in I}\left(r^p \int_{\mathbb{R}^{N} \backslash B_r(x_{0})} \frac{|v(x, t)|^{p-1}}{|x-x_0|^{N+sp}}\,dx\right)^{\frac{1}{p-1}},
\end{align}
where $(x_0, t_0)\in\mathbb{R}^N\times(0,T)$ and the interval $I=\left[t_0-T_1, t_0+T_2\right] \subseteq(0,T)$. This is a parabolic counterpart to the tail introduced in \cite{DKP16}. It is easy to check that $\mathrm{Tail}_\infty(v;x_0,r,I)$ is well-defined for any $v \in L^{\infty}(I; L_{sp}^{p-1}(\mathbb{R}^{N}))$.

\begin{definition}
	\label{def-1-1}
	A function $u\in L^p(I;W_{\rm {loc}}^{s,p}(\Omega)) \cap C(I;L_{\rm {loc}}^2(\Omega)) \cap L^{\infty}(I;L_{sp}^{p-1}(\mathbb{R}^N))$ is a local weak subsolution (super-) to \eqref{1.1} if for any closed interval $I:=[t_1, t_2] \subseteq(0,T)$, there holds that
	\begin{align}
	\label{1.6}
	&\int_\Omega u(x,t_2)\varphi(x,t_2)\,dx-\int_\Omega u(x,t_1)\varphi(x,t_1)\,dx
	-\int_{t_1}^{t_2} \int_\Omega u(x,t)\partial_t \varphi(x,t)\,dxdt\nonumber\\
	&\quad+\int_{t_1}^{t_2}\int_\Omega|\nabla u|^{p-2}\nabla u\cdot \nabla\varphi\,dxdt +\int_{t_1}^{t_2} \mathcal{E}(u,\varphi,t)\,dt \leq(\geq) 0,
	\end{align}
	for every nonnegative test function $\varphi \in L^p (I;W^{s,p}(\Omega))\cap W^{1,2}(I;L^2(\Omega))$ with the property that $\varphi$ has spatial support compactly contained in $\Omega$, where
	\begin{align*}
	\mathcal{E}(u,\varphi,t):=\frac{1}{2} \int_{\mathbb{R}^N} \int_{\mathbb{R}^N}&\Big[|u(x,t)-u(y,t)|^{p-2}(u(x,t)-u(y,t))\\
	&(\varphi(x,t)-\varphi(y,t))K(x,y,t)\Big]\,dxdy.
	\end{align*}
	A function $u$ is a local weak solution to \eqref{1.1} if and only if $u$ is a local weak subsolution and supersolution.
\end{definition}

We now are in a position to state the main contribution of this work. First, we provide the local boundedness of weak solutions in the cases that $p>\frac{2N}{N+2}$ and $1<p\leq \frac{2N}{N+2}$. For two real number, set
\begin{align*}
a \vee b:=\max \{a,b\}, \quad a_+:=\max \{a,0\}, \quad a_-:=-\min \{a,0\}.
\end{align*}

\begin{theorem} [Local boundedness]
	\label{thm-1-2}
	Let $p> 2N/(2+N)$ and $q:=\max\{p,2\}$. Assume that $u$ is a local weak subsolution to \eqref{1.1}. Let $(x_0,t_0) \in Q_T$, $R \in(0,1)$ and $Q_R^- \equiv B_R(x_0) \times(t_0-R^p,t_0)$ such that $\overline{B}_R(x_0)\subseteq\Omega$ and $[t_0-R^p, t_0] \subseteq(0,T)$. Then it holds that
	\begin{align*}
	\operatorname*{ess\,\sup}_{Q_{R/2}^-} u\leq \mathrm{Tail}_\infty\left(u_+;x_0,R/2,t_0-R^p, t_0\right)+C\left(\mint_{Q_R^-} u_+^q\,dx dt\right)^{\frac{p}{N(p\kappa-q)}} \vee 1,
	\end{align*}
	where $\kappa:=1+2/N$ and $C>0$ only depends on $N,p,s$ and $\Lambda$.
\end{theorem}

	In the scenario that $1<p\leq2N/(N+2)$, assuming that the weak subsolution have the following constructions: for $m>\max\left\{2,\frac{N(2-p)}{p}\right\}$, there exists a sequence of $\{u_k\}_{k\in \mathbb{N}}$ whose components are
bounded subsolutions of \eqref{1.1} such that
	\begin{align}
	\label{1.8}
	\|u_k\|_{L_{\rm{loc}}^\infty(0,T;L_{sp}^{p-1}(\mathbb{R}^N))}\leq C
	\end{align}
	and
	\begin{align}
	\label{1.9}
	u_k\rightarrow u \quad \text{in } L_{\rm{loc}}^m(Q_T) \text{ as } k\rightarrow\infty.
	\end{align}

\begin{theorem} [Local boundedness]
	\label{thm-1-4}
	Let $1<p\leq 2N/(N+2)$, $\kappa=1+2/N$ and $m>\max\left\{2,\frac{N(2-p)}{p}\right\}$. Suppose that $u\in L_{\rm{loc}}^m(Q_T)$ with the properties \eqref{1.8} and \eqref{1.9} is a local weak subsolution of \eqref{1.1}. Let $(x_0,t_0) \in Q_T$, $R \in(0,1)$ and $Q_R^- \equiv B_R(x_0) \times(t_0-R^p,t_0)$ such that $\overline{B}_R(x_0)\subseteq\Omega$ and $[t_0-R^p, t_0] \subseteq(0,T)$. Then it holds that
	\begin{align*}
	\operatorname*{ess\,\sup}_{Q_{R/2}^-} u&\leq \mathrm{Tail}_\infty\left(u_+;x_0,R/2,t_0-R^p, t_0\right)\\
	&\quad+C\left(\mint_{Q_R^-} u_+^m\,dx dt\right)^{\frac{p}{(N+p)(m-2-\beta)}} \vee \left(\mint_{Q_R^-} u_+^m\,dx dt\right)^{\frac{p}{(N+p)(m-p-\beta)}},
	\end{align*}
	where $\beta=N(m-p\kappa)/(p+N)$ and $C>0$ only depends on $N,p,s,m$ and $\Lambda$.
\end{theorem}

Based on the boundedness result (Theorem \ref{thm-1-2}), we are able to deduce that the weak solutions are locally H\"{o}lder continuous for $p>2$.
\begin{theorem} [H\"{o}lder continuity]
	\label{thm-1-5}
	Let $p>2$.  Assume that $u$ is a local weak solution of \eqref{1.1}. Let $(x_0,t_0) \in Q_T$, $R \in(0,1)$ and $Q_R \equiv B_R(x_0) \times(t_0-R^p,t_0+R^p)$ such that $\overline{Q}_R\subseteq Q_T$. Moreover, there exists a constant $\alpha\in(0,p/(p-1))$ such that for every $\rho\in(0,R/2)$,
	\begin{align*}
	\operatorname*{ess\,osc}_{Q_{\rho, d \rho^p}} u<C\left(\frac{\rho}{R}\right)^\alpha\left[\mathrm{Tail}_\infty\left(u;x_0,R/2, t_0-R^p, t_0+R^p\right)+\left(\mint_{Q_R}|u|^p\,dxdt\right)^{\frac{1}{2}} \vee 1\right]
	\end{align*}
	with $d\in(0,1)$ and $C\geq 1$ depending on $N,p,s$ and $\Lambda$.
\end{theorem}

The paper is organized as follows. In Section \ref{sec2}, we gather some notations and auxiliary inequalities. Necessary energy estimates are showed in Section \ref{sec3}. Sections \ref{sec4} and \ref{sec5} are devoted to
proving the local boundedness and H\"{o}lder regularity of weak solutions, respectively.

\section{Preliminaries}
\label{sec2}

In this section, we first give some notations for clarity and then provide some important inequalities to be used later.

\subsection{Notation}

Let $B_\rho(x)$ be the open ball with radius $r$ and centered at $x\in\mathbb{R}^N$. We denote the parabolic cylinders by $Q_{\rho,r}(x, t):=B_\rho(x) \times(t-r,t+r)$, $Q_\rho(x,t):=Q_{\rho,\rho^p}(x,t)=B_\rho(x) \times(t-\rho^p,t+\rho^p)$ and $Q_\rho^-(x,t):=Q_{\rho,\rho^p}^-(x, t)=B_\rho(x) \times(t-\rho^p,t)$ with $r,\rho>0$ and $(x,t) \in \mathbb{R}^{N} \times(0,T)$. If not important, or clear from the context, we simply write these symbols by $B_\rho=B_\rho(x)$, $Q_{\rho,r}=Q_{\rho,r}(x,t)$, $Q_\rho=Q_\rho(x,t)$ and $Q_\rho^-=Q_\rho^-(x,t)$. Moreover for $g\in L^1(V)$, we denote the integral average of $g$ by
\begin{align*}
(g)_{V}:=\mint_{V} g(x)\,dx:=\frac{1}{|V|} \int_{V} g(x)\,dx.
\end{align*}
Define
\begin{align*}
J_{p}(a, b)=|a-b|^{p-2}(a-b)
\end{align*}
for any $a, b \in \mathbb{R}$.
We also use the notation
\begin{align*}
d\mu=d\mu(x,y,t)=K(x,y,t)\,dxdy.
\end{align*}
It is worth mentioning that the constant $C$ represents a general positive constant which may differ from each other.

Next, we will show several fundamental but very useful Sobolev inequalities.

\subsection{Sobolev inequalities}

\begin{lemma}
	\label{lem-2-1}
	Let $1 \leq p$, $\ell\leq q<\infty$ satisfy $\frac{N}{p}-\frac{N}{q}\leq 1$ and
	\begin{align*}
	\theta\left(1-\frac{N}{p}+\frac{N}{q}\right)+(1-\theta)\left(\frac{N}{q}-\frac{N}{\ell}\right)=0
	\end{align*}
	with $\theta\in(0,1)$. Then there exists a constant $C>0$ only depending on $N, p, q, \ell$ such that
	\begin{align}
	\label{2.1}
	\|u\|_{L^q(B_1)} \leq C\|Du\|_{L^p(B_1)}^\theta \|u\|_{L^\ell(B_1)}^{1-\theta}
	\end{align}
	for all $u \in W^{1,p}(B_1) \cap L^\ell(B_1)$.
\end{lemma}

\begin{lemma}
	\label{lem-2-2}
	Let $1<p<N$. Then for every $u\in W^{1,p}(B_1)$, there holds that
	\begin{align*}
	\|u\|_{L^{\frac{Np}{N-p}}(B_1)}\leq C\|u\|_{W^{1,p}(B_1)},
	\end{align*}
	where $C>0$ only depends on $N$ and $p$.
\end{lemma}

\begin{lemma}
	\label{lem-2-3}
	Let $0<t_1<t_2$ and $p\in(1,\infty)$. Then for every
	\begin{align*}
	u\in L^p\left(t_1,t_2;W^{1,p}(B_r)\right) \cap L^\infty\left(t_1,t_2;L^2(B_r)\right),
	\end{align*}
	it holds that
	\begin{align}
	\label{2.2}
	&\quad\int_{t_1}^{t_2} \mint_{B_r}|u(x,t)|^{p(1+\frac{2}{N})}\,dxdt \nonumber\\
	&\leq C\left(r^p \int_{t_1}^{t_2}\mint_{B_r} |\nabla u(x,t)|^p\,dxdt+\int_{t_1}^{t_2} \mint_{B_r}|u(x,t)|^p\,dxdt\right) \nonumber\\
	&\quad \times\left(\operatorname*{ess\,\sup}_{t_1<t<t_2} \mint_{B_r}|u(x, t)|^2\,dx\right)^{\frac{p}{N}},
	\end{align}
	where $C>0$ only depends on $p$ and $N$.
\end{lemma}

\begin{proof}
	We divide this proof into two cases.
	
	\textbf{Case 1:} $1<p<N$. Applying Lemma \ref{lem-2-2} and H\"{o}lder inequality, we infer that
	\begin{align*}
	&\quad\int_{t_1}^{t_2} \mint_{B_1}|u(x,t)|^{p(1+\frac{2}{N})}\, dxdt\nonumber\\ &=\int_{t_1}^{t_2} \mint_{B_1}|u(x,t)|^{\frac{2p}{N}}|u(x,t)|^p\, dxdt \nonumber\\
	&\leq \int_{t_1}^{t_2}\left(\mint_{B_1}|u(x,t)|^{2}\, dx\right)^{\frac{ p}{N}}\left(\mint_{B_1}|u(x,t)|^{\frac{pN}{N-p}}\, dx\right)^{\frac{N-p}{N}}\, dt\nonumber\\
	&\leq C\left(\operatorname*{ess\,\sup}_{t_1<t<t_2} \mint_{B_1}|u(x,t)|^{2}\, d x\right)^{\frac{p}{N}} \int_{t_1}^{t_2} \mint_{B_1}|u(x,t)|^p+|\nabla u(x,t)|^p\,dx dt.
	\end{align*}
	By scaling argument, we get
	\begin{align*}
	&\quad\int_{t_1}^{t_2} \mint_{B_r}|u(x,t)|^{p(1+\frac{2}{N})}\,dxdt \nonumber\\
	&\leq C\left(r^p \int_{t_1}^{t_2}\mint_{B_r} |\nabla u(x,t)|^p\,dxdt+\int_{t_1}^{t_2} \mint_{B_r}|u(x,t)|^p\,dxdt\right) \nonumber\\
	&\quad \times\left(\operatorname*{ess\,\sup}_{t_1<t<t_2} \mint_{B_r}|u(x, t)|^2\,dx\right)^{\frac{p}{N}}.
	\end{align*}
	\textbf{Case 2:} $p\geq N$. We can easily find that
	\begin{align*}
	1\geq\frac{N}{p}-\frac{N}{p(1+\frac{2}{N})}, \quad p\left(1+\frac{2}{N}\right)>2
	\end{align*}
	and
	\begin{align*}
	\theta\left(1-\frac{N}{p}+\frac{N}{p\left(1+\frac{2}{N}\right)}\right)+(1-\theta)\left(\frac{N}{p\left(1+\frac{2}{N}\right)}-\frac{N}{2}\right)=0
	\end{align*}
	with $\theta=\frac{N}{N+2}$. Thus by Lemma \ref{lem-2-1}, it follows that
	\begin{align*}
	\|u\|_{L^{p\left(1+\frac{2}{N}\right)}(B_1)} \leq C\|Du\|_{L^p(B_1)}^{\frac{N}{N+2}} \|u\|_{L^2(B_1)}^{\frac{2}{N+2}}
	\end{align*}
	for any $t\in(t_1,t_2)$. Using scaling argument, we have
	\begin{align*}
	\mint_{B_r}|u(x,t)|^{p(1+\frac{2}{N})}\,dx \leq C r^p \mint_{B_r}|\nabla u(x,t)|^p\,dx\times\left(\mint_{B_r}|u(x,t)|^2\,dx\right)^{\frac{p}{N}}
	\end{align*}
	for any $t\in(t_1,t_2)$. Integrating the above inequality over $(t_1,t_2)$, we get
	\begin{align*}
	&\quad\int_{t_1}^{t_2}\mint_{B_r}|u(x,t)|^{p(1+\frac{2}{N})}\,dxdt\\
	&\leq C r^p \int_{t_1}^{t_2}\mint_{B_r}|\nabla u(x,t)|^p\,dx\times\left(\mint_{B_r}|u(x,t)|^2\,dx\right)^{\frac{p}{N}}\,dt\\
	&\leq C r^p \int_{t_1}^{t_2}\mint_{B_r}|\nabla u(x,t)|^p\,dxdt\times\left(\operatorname*{ess\,\sup}_{t_1<t<t_2}\mint_{B_r}|u(x,t)|^2\,dx\right)^{\frac{p}{N}}.
	\end{align*}
\end{proof}

\begin{lemma}
	\label{lem-2-4}
	Let $0<t_1<t_2$ and $p\in(1,\infty)$. Then for every
	\begin{align*}
	u\in L^p\left(t_1,t_2;W^{1,p}(B_r)\right) \cap L^\infty\left(t_1,t_2;L^p(B_r)\right),
	\end{align*}
	it holds that
	\begin{align}
	\label{2.3}
	&\quad\int_{t_1}^{t_2} \mint_{B_r}|u(x,t)|^{p(1+\frac{p}{N})}\,dxdt \nonumber\\
	&\leq C\left(r^p \int_{t_1}^{t_2}\mint_{B_r} |\nabla u(x,t)|^p\,dxdt+\int_{t_1}^{t_2} \mint_{B_r}|u(x,t)|^p\,dxdt\right) \nonumber\\
	&\quad \times\left(\operatorname*{ess\,\sup}_{t_1<t<t_2} \mint_{B_r}|u(x, t)|^p\,dx\right)^{\frac{p}{N}},
	\end{align}
	where $C>0$ only depends on $p$ and $N$.	
\end{lemma}
\begin{proof}
	It follows from Lemma \ref{lem-2-1} that
	\begin{align*}
	\|u\|_{L^{p\left(1+\frac{p}{N}\right)}(B_1)} \leq C\|Du\|_{L^p(B_1)}^{\frac{N}{N+p}} \|u\|_{L^p(B_1)}^{\frac{p}{N+p}}
	\end{align*}
	for all $t\in (t_1,t_2)$, where $C>0$ only depends on $p$ and $N$. Using scaling argument, we have
	\begin{align*}
	\mint_{B_r}|u(x,t)|^{p(1+\frac{p}{N})}\,dx \leq C r^p \mint_{B_r}|\nabla u(x,t)|^p\,dx\times\left(\mint_{B_r}|u(x,t)|^p\,dx\right)^{\frac{p}{N}}.
	\end{align*}
	Integrating the above inequality over $(t_1,t_2)$, we get
	\begin{align*}
	&\quad\int_{t_1}^{t_2}\mint_{B_r}|u(x,t)|^{p(1+\frac{p}{N})}\,dxdt\\
	&\leq C r^p \int_{t_1}^{t_2}\mint_{B_r}|\nabla u(x,t)|^p\,dx\times\left(\mint_{B_r}|u(x,t)|^p\,dx\right)^{\frac{p}{N}}\,dt\\
	&\leq C r^p \int_{t_1}^{t_2}\mint_{B_r}|\nabla u(x,t)|^p\,dxdt\times\left(\operatorname*{ess\,\sup}_{t_1<t<t_2}\mint_{B_r}|u(x,t)|^p\,dx\right)^{\frac{p}{N}}.
	\end{align*}
\end{proof}

\section{Energy Estimates}
\label{sec3}
In this section, we will establish the Caccioppoli inequality and Logarithmic form inequality for mixed local and nonlocal type equation \eqref{1.1}.  The first step of the proof should be the regularization procedure with respect to the time variable, which can be
performed by straightforward adaptation of standard reasonings as used in \cite{BLS,DZZ21}. We omit this step here.

\begin{lemma} [Caccioppoli-type inequality]
\label{lem-3-1}
Let $p>1$ and $u$ be a local subsolution of \eqref{1.1}. Let $B_r \equiv B_r(x_0)$ satisfy $\overline{B}_r \subseteq \Omega$ and $0<\tau_1<\tau_2$, $\ell>0$ satisfy $\left[\tau_1-\ell, \tau_2\right] \subseteq(0,T).$ For any nonnegative functions $\psi \in C_0^\infty (B_r)$ and $\eta \in C^\infty(\mathbb{R})$ such that $\eta(t) \equiv 0$ if $t \leq \tau_1-\ell$ and $\eta(t) \equiv 1$ if $t \geq \tau_1$, it holds that
\begin{align}
\label{3.1}
&\quad\int_{\tau_1-\ell}^{\tau_2} \int_{B_r} |\nabla w_+(x,t)|^p\psi^p(x)\eta^2(t)\,dxdt+\operatorname*{ess\,\sup}_{\tau_1<t<\tau_2}\int_{B_r} w_+^2(x,t) \psi^p(x)\,dx\nonumber\\
&\quad+\int_{\tau_1-\ell}^{\tau_2} \int_{B_r} \int_{B_r}|w_+(x,t)\psi(x)-w_+(y,t)\psi(y)|^p \eta^2(t)\, d\mu dt\nonumber\\
&\leq C \int_{\tau_1-\ell}^{\tau_2} \int_{B_r} |\nabla \psi(x)|^p w_+^p(x,t)\eta^2(t)\,dxdt\nonumber\\
&\quad+ C \int_{\tau_1-\ell}^{\tau_2} \int_{B_r} \int_{B_r} (\max\{w_+(x,t), w_+(y,t)\})^p|\psi(x)-\psi(y)|^p \eta^2(t)\, d\mu dt\nonumber\\
&\quad+C \mathop{\mathrm{ess}\,\sup}_{\stackrel{\tau_1-\ell<t<\tau_2}{x \in \mathrm{supp}\,\psi}}  \int_{\mathbb{R}^N \backslash B_r} \frac{w_+^{p-1}(y,t)}{|x-y|^{N+sp}}\,dy \int_{\tau_1-\ell}^{\tau_2} \int_{B_r} w_+(x,t) \psi^p(x) \eta^2(t)\, dxdt\nonumber\\
&\quad+C \int_{\tau_1-\ell}^{\tau_2} \int_{B_r} w_+^2(x,t)\psi^p(x)\eta(t)|\partial_t \eta(t)|\,dxdt,
\end{align}
where $w_+(x,t):=\left(u(x,t)-k\right)_+$ and $C>0$ depends on $N,p,s,\Lambda$.
\end{lemma}

\begin{proof}
Taking $\varphi(x,t):=w_+(x,t)\psi^p(x)\eta^2(t)=\left(u(x,t)-k\right)_+\psi^p(x)\eta^2(t)$ to test the weak formulation of \eqref{1.1}. For $s\in\left[\tau_1,\tau_2\right]$, we can obtain
\begin{align*}
0&\geq\int_{\tau_1-\ell}^{s}\int_{B_r}\partial_t u(x,t) w_+(x,t)\psi^p(x)\eta^2(t)\,dxdt\\
&\quad+\int_{\tau_1-\ell}^{s}\int_{B_r}|\nabla u(x,t)|^{p-2}\nabla u(x,t)\cdot\nabla(w_+(x,t)\psi^p(x)\eta^2(t))\,dxdt\\
&\quad+\int_{\tau_1-\ell}^{s}\int_{B_r}\int_{B_r}J_p(w(x,t),w(y,t))\left(w_+(x,t)\psi^p(x)\eta^2(t)-w_+(y,t)\psi^p(y)\eta^2(t)\right)\,d\mu dt\\
&\quad+2\int_{\tau_1-\ell}^{s}\int_{\mathbb{R}^N\backslash B_r}\int_{B_r}J_p\left(w(x,t),w(y,t)\right)w_+(x,t)\psi^p(x)\eta^2(t)\,d\mu dt\\
&=:I_1+I_2+I_3+I_4.
\end{align*}	
Then we are going to estimate $I_1, I_2, I_3$ and $I_4$. First, we evaluate
\begin{align}
\label{3.2}
I_1&=\frac{1}{2}\int_{\tau_1-\ell}^{s}\int_{B_r}\partial_t w_+^2(x,t)\psi^p(x)\eta^2(t)\,dxdt\nonumber\\
&=\frac{1}{2}\int_{B_r}w_+^2(x,t)\psi^p(x)\eta^2(t)dx\bigg|_{\tau_1-\ell}^s-\int_{\tau_1-\ell}^{s}\int_{B_r}w_+^2(x,t)\psi^p(x)\eta(t)\partial_t\eta(t)\,dxdt\nonumber\\
&=\frac{1}{2}\int_{B_r}w_+^2(x,s)\psi^p(x)\,dx-\int_{\tau_1-\ell}^{s}\int_{B_r}w_+^2(x,t)\psi^p(x)\eta(t)\partial_t\eta(t)\,dxdt,
\end{align}
where in the last line we note that $\eta(\tau_1-\ell)=0$ and $\eta(s)=1$ when $s\geq\tau_1$.
We produce with estimating $I_2$. By Young's inequality with $\varepsilon$, it yields that
\begin{align}
\label{3.3}
I_2&=\int_{\tau_1-\ell}^{s}\int_{B_r}\psi^p\eta^2|\nabla u|^{p-2}\nabla u\cdot\nabla w_++p\psi^{p-1}\eta^2w_+|\nabla u|^{p-2}\nabla u\cdot\nabla\psi\,dxdt\nonumber\\
&\geq \int_{\tau_1-\ell}^{s}\int_{B_r}|\nabla w_+|^{p}\psi^p\eta^2\,dxdt-\int_{\tau_1-\ell}^{s}\int_{B_r}p\psi^{p-1}\eta^2 w_+|\nabla u|^{p-1}|\nabla\psi|\,dxdt\nonumber\\
&\geq \int_{\tau_1-\ell}^{s}\int_{B_r}|\nabla w_+|^{p}\psi^p\eta^2\,dxdt-\varepsilon\int_{\tau_1-\ell}^{s}\int_{B_r}\psi^{p}\eta^2|\nabla w_+|^p\,dxdt\nonumber\\
&\quad-C(p,\varepsilon)\int_{\tau_1-\ell}^{s}\int_{B_r}|\nabla \psi|^p\eta^2 w_+^p\,dxdt\nonumber\\
&=\frac{1}{2}\int_{\tau_1-\ell}^{s}\int_{B_r}|\nabla w_+|^{p}\psi^p\eta^2\,dxdt-C(p)\int_{\tau_1-\ell}^{s}\int_{B_r}|\nabla \psi|^p\eta^2 w_+^p\,dxdt,
\end{align}
with $\varepsilon=\frac{1}{2}$. As the same proof of Lemma 3.3 in \cite{DZZ21}, we have
\begin{align}
\label{3.4}
I_3&\geq\frac{1}{2^{p+1}}\int_{\tau_1-\ell}^{s}\int_{B_r}\int_{B_r}|w_+(x,t)\psi(x)-w_+(y,t)\psi(y)|^p\eta^2(t)\,d\mu dt\nonumber\\
&\quad-C\int_{\tau_1-\ell}^{s}\int_{B_r}\int_{B_r}(\max\{w_+(x,t),w_+(y,t)\})^p|\psi(x)-\psi(y)|^p\eta^2(t)\,d\mu dt
\end{align}
and
\begin{align}
\label{3.5}
I_4\geq -C \mathop{\mathrm{ess}\,\sup}_{\stackrel{\tau_1-\ell<t<s}{x \in \mathrm{supp}\,\psi}}\int_{\mathbb{R}^N \backslash B_r}\frac{w_+^{p-1}(y,t)}{|x-y|^{N+sp}}\,dy\int_{\tau_1-\ell}^{s}\int_{B_r}w_+(x,t)\psi^p(x)\eta^2(t)\,dxdt.
\end{align}
Merging the estimates on \eqref{3.2}--\eqref{3.5}, we get
\begin{align*}
&\quad\int_{\tau_1-\ell}^{s} \int_{B_r} |\nabla w_+(x,t)|^p\psi^p(x)\eta^2(t)\,dxdt+\int_{B_r} w_+^2(x,s) \psi^p(x)\,dx\nonumber\\
&\quad+\int_{\tau_1-\ell}^{s} \int_{B_r} \int_{B_r}|w_+(x,t)\psi(x)-w_+(y,t)\psi(y)|^p \eta^2(t)\, d\mu dt\nonumber\\
&\leq C \int_{\tau_1-\ell}^{\tau_2} \int_{B_r} |\nabla \psi(x)|^p w_+^p(x,t)\eta^2(t)\,dxdt\nonumber\\
&\quad+ C \int_{\tau_1-\ell}^{\tau_2} \int_{B_r} \int_{B_r} \max \left\{w_+(x,t), w_+(y,t)\right\}^p|\psi(x)-\psi(y)|^p \eta^2(t)\, d\mu dt\nonumber\\
&\quad+C \mathop{\mathrm{ess}\,\sup}_{\stackrel{\tau_1-\ell<t<\tau_2}{x \in \mathrm{supp}\, \psi}} \int_{\mathbb{R}^N \backslash B_r} \frac{w_+^{p-1}(y,t)}{|x-y|^{N+sp}}\,dy \int_{\tau_1-\ell}^{\tau_2} \int_{B_r} w_+(x,t) \psi^p(x) \eta^2(t)\, dxdt\nonumber\\
&\quad+C \int_{\tau_1-\ell}^{\tau_2} \int_{B_r} w_+^2(x,t)\psi^p(x)\eta(t)|\partial_t \eta(t)|\,dxdt,
\end{align*}
which leads to the desired result.
\end{proof}

From the forthcoming lemma, one can interpret the reason why we only derive the H\"{o}lder continuity of weak solutions in the case $p>2$. For the subquadratic case, we at present can not show a similar logarithmic estimate.

\begin{lemma} [Logarithmic estimates]
\label{lem-3-2}
Let $p>2$ and $u$ be a local solution to \eqref{1.1}. Let $B_r\equiv B_r(x_0)$ and $(x_0,t_0)\in Q_{T}, T_0>0$, $0<r\leq R/2$. We also note that $\tilde{Q} \equiv B_R(x_0) \times(t_0-2T_0, t_0+2T_0)$ such that $\overline{B}_{R}(x_0) \subseteq \Omega$ and
$\left[t_0-2T_0,t_0+2T_0\right] \subseteq(0,T)$. If $u\in L^\infty(\tilde{Q})$ and $u\geq 0$ in $\tilde{Q}$. For any $d>0$, it holds that
\begin{align}
\label{3.6}
&\quad\int_{t_0-T_0}^{t_0+T_0} \int_{B_r}|\nabla \log\left(u(x,t)+d\right)|^p\,dxdt+\int_{t_0-T_0}^{t_0+T_0} \int_{B_r}\int_{B_r} \left|\log \left(\frac{u(x,t)+d}{u(y,t)+d}\right)\right|^p\,d\mu dt\nonumber\\
&\leq  C T_0 r^Nd^{1-p}R^{-p}\left[\mathrm{Tail}_\infty(u;x_0,R,t_0-2 T_0,t_0+2 T_0)\right]^{p-1}\nonumber\\
&\quad+Cr^Nd^{2-p}+CT_0r^{N-sp}+CT_0r^{N-p},
\end{align}
where $C>0$ depends on $N,p,s$ and $\Lambda$.
\end{lemma}

\begin{proof}
Let $\psi\in C_0^\infty(B_{3r/2})$ and $\eta\in C_0^\infty(t_0-2T_0,t_0+2T_0)$ fulfill
\begin{align*}
0\leq\psi\leq 1,\quad|\nabla \psi|<Cr^{-1}  \text{ in } B_{2r},\quad \psi\equiv 1 \text{ in } B_r,
\end{align*}
and
\begin{align*}
0\leq\eta\leq 1,\quad|\partial_{t}\eta|<CT_0^{-1}  \text{ in } (t_0-2T_0,t_0+2T_0),\quad \eta\equiv 1\text{ in } (t_0-T_0,t_0+T_0).
\end{align*}
Choosing $\varphi(x,t):=(u(x,t)+d)^{1-p} \psi^p(x) \eta^2(t)$ to text the weak formulation of \eqref{1.1}, we get
\begin{align*}
0&=-\int_{t_0-2T_0}^{t_0+2T_0} \int_{B_{2r}} \partial_{t}\left((u(x,t)+d)^{1-p} \psi^p(x) \eta^2(t)\right)u(x,t)\,dxdt\\
&\quad+\int_{t_0-2T_0}^{t_0+2T_0} \int_{B_{2r}}|\nabla u|^{p-2}\nabla u\cdot\nabla \left((u(x,t)+d)^{1-p} \psi^p(x) \eta^2(t)\right)\,dxdt\\
&\quad+\int_{t_0-2T_0}^{t_0+2T_0} \int_{B_{2r}}\int_{B_{2r}} J_p(u(x,t),u(y,t))\left[\frac{\psi^p(x)}{(u(x,t)+d)^{p-1}}-\frac{\psi^p(y)}{(u(y, t)+d)^{p-1}}\right] \eta^2(t)\, d\mu dt\\
&\quad+2\int_{t_0-2T_0}^{t_0+2T_0}\int_{\mathbb{R}^{N} \backslash B_{2r}}\int_{B_{2r}}J_p(u(x,t),u(y,t))\frac{\psi^p(x)}{(u(x,t)+d)^{p-1}}\eta^2(t)\, d\mu dt\\
&=:I_1+I_2+I_3+I_4.
\end{align*}
It follows from the proof of \cite[Lemma 3.5]{DZZ21} that
\begin{align}
\label{3.7}
I_1&=\int_{t_0-2T_0}^{t_0+2T_0} \int_{B_{2r}} \left((u(x,t)+d)^{1-p} \psi^p(x) \eta^2(t)\right)\partial_{t}u(x,t)\,dxdt\nonumber\\
&\leq Cr^N d^{2-p},
\end{align}
\begin{align}
\label{3.8}
I_3 \leq-C \int_{t_0-T_0}^{t_0+T_0} \int_{B_r} \int_{B_r}\left|\log \left(\frac{u(x,t)+d}{u(y,t)+d}\right)\right|^p\,dxdydt+C T_0 r^{N-sp},
\end{align}
and
\begin{align}
\label{3.9}
I_4 \leq C T_0 r^{N-sp}+C T_0 r^N R^{-p}d^{1-p}\left[\mathrm{Tail}_\infty\left(u;x_0,R,t_0-2T_0,t_0+2T_0\right)\right]^{p-1}.
\end{align}
By Young's inequality with $\varepsilon$, we estimate the integral $I_2$ as
\begin{align}
\label{3.10}
I_2&=-(p-1)\int_{t_0-2T_0}^{t_0+2T_0} \int_{B_{2r}}(u(x,t)+d)^{-p}\psi^p(x)|\nabla u(x,t)|^p\eta^2(t)\,dxdt\nonumber\\
&\quad+p\int_{t_0-2T_0}^{t_0+2T_0} \int_{B_{2r}}(u(x,t)+d)^{1-p}\psi^{p-1}(x)|\nabla u(x,t)|^{p-2}\nabla u(x,t)\cdot \nabla\psi(x)\eta^2(t)\,dxdt\nonumber\\
&\leq-(p-1)\int_{t_0-2T_0}^{t_0+2T_0} \int_{B_{2r}}(u(x,t)+d)^{-p}\psi^p(x)|\nabla (u(x,t)+d)|^p\eta^2(t)\,dxdt\nonumber\\
&\quad+\varepsilon\int_{t_0-2T_0}^{t_0+2T_0} \int_{B_{2r}}(u(x,t)+d)^{-p}\psi^p(x)|\nabla (u(x,t)+d)|^p\eta^2(t)\,dxdt\nonumber\\
&\quad+C(\varepsilon)\int_{t_0-2T_0}^{t_0+2T_0} \int_{B_{2r}}|\nabla\psi(x)|^p\eta^2(t)\,dxdt\nonumber\\
&\leq -C(p)\int_{t_0-2T_0}^{t_0+2T_0} \int_{B_r}(u(x,t)+d)^{-p}|\nabla (u(x,t)+d)|^p\eta^2(t)\,dxdt+C(N,p)T_0 r^{N-p}\nonumber\\
&\leq-C(p)\int_{t_0-T_0}^{t_0+T_0} \int_{B_r}|\nabla\log(u(x,t)+d)|^p\,dxdt+C(N,p)T_0 r^{N-p}
\end{align}
with $\varepsilon$ satisfying $\varepsilon<p-1$.
Combining with \eqref{3.7}--\eqref{3.10}, we can obtain the Logarithmic estimates.
\end{proof}
Next, we will give a corollary of Lemma \ref{lem-3-2}, which plays a crucial role in obtaining the H\"{o}lder continuity.
\begin{corollary}
\label{cor-3-3}
Let $p>2$ and $u$ be a local solution to \eqref{1.1}. Let $B_r\equiv B_r(x_0)$ and $(x_0,t_0)\in Q_{T}, T_0>0$, $0<r\leq R/2$. We also note that $\tilde{Q} \equiv B_R(x_0) \times(t_0-2T_0, t_0+2T_0)$ such that $\overline{B}_{R}(x_0) \subseteq \Omega$ and
$\left[t_0-2T_0,t_0+2T_0\right] \subseteq(0,T)$. Suppose that $u\in L^\infty(\tilde{Q})$ and $u\geq 0$ in $\tilde{Q}$. Let $a,d>0$, $b>1$ and define
\begin{align*}
v:=\min \left\{(\log(a+d)-\log(u+d))_+, \log b\right\}.
\end{align*}
Then it holds that
\begin{align}
\label{3.11}
&\quad\int_{t_0-T_0}^{t_0+T_0}\mint_{B_r}\left|v(x,t)-(v)_{B_r}(t)\right|^p\,dxdt\nonumber\\
&\leq C T_0 d^{1-p}\left(\frac{r}{R}\right)^p\left[\mathrm{Tail}_\infty\left(u;x_0, R,t_0-2 T_0,t_0+2 T_0\right)\right]^{p-1}\nonumber\\
&\quad+CT_0+Cd^{2-p}r^p+CT_0r^{p-sp},
\end{align}
where $C>0$ depends on $N,p,s$ and $\Lambda$.
\end{corollary}

\begin{proof}
By the Poincar\'{e} inequality from Theorem 2 in \cite{E98}, it yields that
\begin{align*}
\mint_{B_r}\left|v(x,t)-(v)_{B_r}(t)\right|^p\,dx\leq Cr^{p-N}\int_{B_r}|\nabla v(x,t)|^p\,dx
\end{align*}
for any $t\in\left(t_0-T_0,t_0+T_0\right)$. Integrating the above inequality over $\left(t_0-T_0,t_0+T_0\right)$ leads to
\begin{align}
\label{3.12}
\int_{t_0-T_0}^{t_0+T_0}\mint_{B_r}\left|v(x,t)-(v)_{B_r}(t)\right|^p\,dxdt\leq Cr^{p-N}\int_{t_0-T_0}^{t_0+T_0}\int_{B_r}|\nabla v(x,t)|^p\,dxdt,
\end{align}
where $C>0$ depends only on $N,p$. Observing that $v$ is a truncation function of the sum of a constant and $\log(u+d)$, which gives that
\begin{align}
\label{3.13}
\int_{t_0-T_0}^{t_0+T_0}\int_{B_r}|\nabla v(x,t)|^p\,dxdt\leq \int_{t_0-T_0}^{t_0+T_0}\int_{B_r}|\nabla \log(u(x,t)+d)|^p\,dxdt.
\end{align}
We can get the results from Lemma \ref{lem-3-2} along with \eqref{3.12} and \eqref{3.13}.
\end{proof}

\section{Local boundedness}
\label{sec4}

In this part, we are ready to study the local boundedness of weak solutions. To this end, we first introduce some notations. For $\sigma\in[1/2,1)$, set
\begin{align*}
r_0:=r, \quad r_j:=\sigma r+2^{-j}(1-\sigma) r, \quad \tilde{r}_j:=\frac{r_j+r_{j+1}}{2}, \quad j=0,1,2, \ldots
\end{align*}
and
\begin{align*}
Q_j^-:=B_j \times \Gamma_j:=B_{r_j}(x_0) \times\left(t_0-r_j^p,t_0\right), \quad j=0,1,2, \ldots,
\end{align*}
\begin{align*}
\tilde{Q}_j^-:=\tilde{B}_j\times \tilde{\Gamma}_j:=B_{\tilde{r}_j}(x_{0}) \times\left(t_0-\tilde{r}_j^p,t_0\right), \quad j=0,1,2, \ldots .
\end{align*}
Denote
\begin{align*}
k_j:=\left(1-2^{-j}\right)\tilde{k}, \quad \tilde{k}_j:=\frac{k_{j+1}+k_j}{2}, \quad j=0,1,2, \ldots
\end{align*}
with
\begin{align*}
\tilde{k} \geq \frac{\mathrm{Tail}_\infty\left(u_+;x_0,\sigma r,t_0-r^p,t_0\right)}{2}.
\end{align*}
Let
\begin{align*}
w_j:=(u-k_j)_+, \quad \tilde{w}_j:=\left(u-\tilde{k}_j\right)_+, \quad j=0,1,2, \ldots.
\end{align*}

We now provide a Caccioppoli type inequality in a special cylinder, which leads to the recursive inequalities.

\begin{lemma}
\label{lem-4-1}
Let $p>1$ and $u$ be a subsolution to \eqref{1.1}. Let $(x_0, t_0) \in Q_T$, $0<r<1$ and $Q_r^-=B_r(x_0) \times(t_0-r^p,t_0)$ such that $\overline{B}_{r}(x_0) \subseteq \Omega$ and $\left[t_0-r^p,t_0\right] \subseteq(0,T)$. Suppose $q$ is a parameter satisfying $q\geq\max\{p,2\}$, it holds that
\begin{align}
\label{4.1}
&\quad\int_{\Gamma_{j+1}}\mint_{B_{j+1}}|\nabla \tilde{w}_j(x,t)|^p\,dxdt+\operatorname*{ess\,\sup}_{t\in\Gamma_{j+1}}\mint_{B_{j+1}} \tilde{w}_j^2(x,t)\, dx\nonumber\\
&\quad+\int_{\Gamma_{j+1}} \int_{B_{j+1}} \mint_{B_{j+1}} \frac{|\tilde{w}_j(x,t)-\tilde{w}_j(y,t)|^p}{|x-y|^{N+sp}}\,dxdydt\nonumber\\
&\leq \frac{C}{r^p}\left(\frac{1}{\sigma^p(1-\sigma)^{N+s p}}+\frac{1}{(1-\sigma)^p}\right)\cdot\left(\frac{2^{(p+q-2) j}}{\tilde{k}^{q-2}}+\frac{2^{(N+sp+q-1)j}}{\tilde{k}^{q-p}}\right) \int_{\Gamma_j} \mint_{B_j} w_j^q(x,t)\,dxdt,
\end{align}
where $C>0$ only depends on $N,p,s,\Lambda$ and $q$.
\end{lemma}

\begin{proof}
First we give a trivial but very useful inequality
\begin{align}
\label{4.2}
\tilde{w}_j^{\tau}(x,t) \leq \frac{C 2^{(q-\tau)j}}{\tilde{k}^{q-\tau}} w_j^q(x,t) \quad \text { in } Q_T,
\end{align}	
where $0\leq\tau< q$. We take the cut-off functions $\psi_j \in C_0^\infty(\tilde{B}_j)$ and $\eta_j \in C_0^\infty(\tilde{\Gamma}_j)$ such that
\begin{align*}
0 \leq \psi_j \leq 1, \quad\left|\nabla \psi_j\right| \leq \frac{C 2^j}{(1-\sigma) r} \text { in } \tilde{B}_j, \quad \psi_j \equiv 1 \text { in } B_{j+1}
\end{align*}
and
\begin{align*}
0 \leq \eta_j \leq 1, \quad\left|\partial_t \eta_j\right| \leq \frac{C 2^{pj}}{(1-\sigma)^pr^p} \text { in } \tilde{\Gamma}_j, \quad \eta_j \equiv 1 \text { in } \Gamma_{j+1}.
\end{align*}
Let $r=r_j$, $\tau_2=t_0$, $\tau_1=t_0-r_{j+1}^p$ and $\ell=\tilde{r}_j^p-r_{j+1}^p$ in Lemma \ref{lem-3-1}. Then we arrive at
\begin{align}
\label{4.3}
&\quad\int_{\tilde{\Gamma}_j} \int_{B_j}|\nabla \tilde{w}_j(x,t)|^p\psi_j^p(x)\eta_j^2(t)\,dxdt+\operatorname*{ess\,\sup}_{t\in\Gamma_{j+1}}\int_{B_j} \tilde{w}_j^2(x,t)\psi_j^p(x)\, dx\nonumber \\
&\quad+\int_{\tilde{\Gamma}_j} \int_{B_j} \int_{B_j} \left|\tilde{w}_j(x,t) \psi_j(x)-\tilde{w}_j(y,t) \psi_j(y)\right|^p\eta_j^2(t)\, d\mu dt\nonumber\\
&\leq C\int_{\tilde{\Gamma}_j} \int_{B_j}|\nabla \psi_j(x)|^p \tilde{w}_j^p(x,t)\eta_j^2(t)\,dxdt\nonumber\\
&\quad+C\int_{\tilde{\Gamma}_j} \int_{B_j} \int_{B_j}(\max\{\tilde{w}_j(x,t), \tilde{w}_j(y,t)\})^p\left|\psi_j(x)-\psi_j(y)\right|^p \eta_j^2(t)\,d\mu dt\nonumber \\
&\quad+C\mathop{\mathrm{ess}\,\sup}_{\stackrel{t\in\tilde{\Gamma}_j}{x \in \mathrm{supp}\, \psi_j}}\int_{\mathbb{R}^N\backslash B_j} \frac{\tilde{w}_j^{p-1}(y,t)}{|x-y|^{N+s p}}\,dy\int_{\tilde{\Gamma}_j} \int_{B_j}\tilde{w}_j(x,t)\psi_j^p(x) \eta_j^2(t)\,d xdt\nonumber\\
&\quad+C\int_{\tilde{\Gamma}_j} \int_{B_j} \tilde{w}_j^2(x,t)\psi_j^p(x) \eta_j(t)\left|\partial_t \eta_j(t)\right|\,dxdt\nonumber\\
&=: I_1+I_2+I_3+I_4.
\end{align}
Using inequality \eqref{4.2} and the definition of $\psi_j$, we estimate $I_1$ as
\begin{align}
\label{4.4}
I_1&\leq \frac{C 2^{pj}}{(1-\sigma)^{p} r^p} \int_{{\Gamma}_j} \int_{B_j} \tilde{w}_j^p(x,t)\,dxdt\nonumber\\
&\leq \frac{C 2^{qj}}{\tilde{k}^{q-p}(1-\sigma)^{p} r^p} \int_{{\Gamma}_j} \int_{B_j} w_j^q(x,t)\,dxdt.
\end{align}
Analogous to the proof of Lemma 4.1 in \cite{DZZ21}, we have
\begin{align}
\label{4.5}
I_2\leq \frac{C 2^{qj}}{\tilde{k}^{q-p}(1-\sigma)^p r^{sp}} \int_{\Gamma_j} \int_{B_j} w_j^q(x,t)\,dxdt,
\end{align}
\begin{align}
\label{4.6}
I_3\leq \frac{C 2^{(N+sp+q-1)j}}{\tilde{k}^{q-p} \sigma^p(1-\sigma)^{N+sp}r^p} \int_{\Gamma_j} \int_{B_j} w_j^q(x,t)\,dxdt
\end{align}
and
\begin{align}
\label{4.7}
I_4\leq \frac{C 2^{(p+q-2)j}}{\tilde{k}^{q-2}(1-\sigma)^p r^p } \int_{\Gamma_j} \int_{B_j} w_j^q(x,t)\,dxdt.
\end{align}
By virtue of the facts that $\psi_j\equiv 1$ in $B_{j+1}$,  $\eta_j \equiv 1$ in $\Gamma_{j+1}$ and \eqref{1.4}, merging inequalities \eqref{4.3}--\eqref{4.7}, we get
\begin{align*}
&\quad\int_{\Gamma_{j+1}}\mint_{B_{j+1}}|\nabla \tilde{w}_j(x,t)|^p\,dxdt+\operatorname*{ess\,\sup}_{t\in\Gamma_{j+1}}\mint_{B_{j+1}} \tilde{w}_j^2(x,t)\, dx\\
&\quad+\int_{\Gamma_{j+1}} \int_{B_{j+1}} \mint_{B_{j+1}} \frac{|\tilde{w}_j(x,t)-\tilde{w}_j(y,t)|^p}{|x-y|^{N+sp}}\,dxdydt\nonumber\\
&\leq \frac{C}{r^p(1-\sigma)^p}\left(\frac{2^{qj}}{r^{(s-1)p}\tilde{k}^{q-p}}+\frac{2^{(N+sp+q-1)j}}{\sigma^p(1-\sigma)^{N+p(s- 1)}\tilde{k}^{q-p}}+\frac{2^{qj}}{\tilde{k}^{q-p}}+\frac{2^{(p+q-2)j}}{\tilde{k}^{q-2}}\right) \nonumber\\ &\quad\times\int_{\Gamma_j} \mint_{B_j} w_j^q(x,t)\,dxdt.
\end{align*}
After rearrangement, we get the desired result.
\end{proof}

The following two lemmas are the consequences of Lemmas \ref{lem-2-3} and \ref{lem-4-1}.

\begin{lemma}
\label{lem-4-2}
Let $p>2N/(N+2)$ and $\max\{p,2\}\leq q<p(N+2)/N$. Let $(x_0,t_0) \in Q_T$, $0<r<1$ and $Q_r^-=B_r(x_0) \times(t_0-r^p,t_0)$ such that $\overline{B}_{r}(x_0) \subseteq \Omega$ and $\left[t_0-r^p,t_0\right] \subseteq(0,T)$. Then for a local subsolution $u$ to \eqref{1.1}, we infer that
\begin{align}
\label{4.8}
&\quad\int_{\Gamma_{j+1}}\mint_{B_{j+1}}w_{j+1}^q(x,t)\,dxdt\nonumber\\
&\leq \frac{C 2^{bj}}{r^{\frac{pq}{\kappa N}}}\left(\frac{1}{\sigma^{\frac{q(N+p)}{\kappa N}}(1-\sigma)^{\frac{q(N+p)(N+sp)}{p \kappa N}}}+\frac{1}{(1-\sigma)^{\frac{q(N+p)}{\kappa N}}}\right)\nonumber\\
&\quad\times\left(\frac{1}{\tilde{k}^{\frac{q}{\kappa}(\frac{q}{N}+1-\frac{2}{p})}}+\frac{1}{\tilde{k}^{\frac{q}{\kappa}(\frac{q}{N}+\frac{2}{N}-\frac{p}{N})}}\right)\left(\int_{\Gamma_j} \mint_{B_j}w_j^q(x,t)\,dxdt\right)^{1+\frac{q}{\kappa N}}
\end{align}
for $j\in\mathbb{N}$, where $b:=(1+p/N)(N+p+q)$, $\kappa:=1+2/N$ and $C>0$ only depends on $N,p,s,\Lambda$ and $q$.
\end{lemma}

\begin{proof}
Since $q<p\kappa$, it follows from H\"{o}lder inequality that
\begin{align}
\label{4.9}
&\quad\int_{\Gamma_{j+1}} \mint_{B_{j+1}} w_{j+1}^q(x,t)\,dxdt \nonumber\\
&\leq \int_{\Gamma_{j+1}} \mint_{B_{j+1}} \tilde{w}_j^q(x,t)\,dxdt\nonumber\\
&\leq\left(\int_{\Gamma_{j+1}} \mint_{B_{j+1}} \tilde{w}_j^{p \kappa}(x,t)\,dxdt\right)^{\frac{q}{p\kappa}}\left(\int_{\Gamma_{j+1}} \mint_{B_{j+1}} \chi_{\{u \geq \tilde{k}_j\}}(x,t)\,dxdt\right)^{1-\frac{q}{p \kappa}}.
\end{align}
From \eqref{4.2}, we can get
\begin{align}
\label{4.10}
\int_{\Gamma_{j+1}} \mint_{B_{j+1}} \chi_{\{u \geq \tilde{k}_j\}}(x,t)\, dxdt \leq \frac{C 2^{qj}}{\tilde{k}^q} \int_{\Gamma_j} \mint_{B_j} w_j^q(x,t)\,dxdt
\end{align}
and
\begin{align}
\label{4.11}
\int_{\Gamma_{j+1}} \mint_{B_{j+1}} \tilde{w}_j^p(x,t)\,dxdt \leq \frac{C 2^{(q-p)j}}{\tilde{k}^{q-p}} \int_{\Gamma_j} \mint_{B_j} w_j^q(x,t)\,dxdt.
\end{align}
Now by \eqref{4.11}, Lemma \ref{lem-2-3} and Lemma \ref{lem-4-1} we estimate
\begin{align}
\label{4.12}
&\quad\int_{\Gamma_{j+1}}\mint_{B_{j+1}}\tilde{w}_j^{p\kappa}(x,t)\,dxdt\nonumber\\
&\leq C
\left(r^p \int_{\Gamma_{j+1}} \mint_{B_{j+1}} |\nabla \tilde{w}_j(x,t)|^p\,dxdt
+\int_{\Gamma_{j+1}} \mint_{B_{j+1}} \tilde{w}_j^p(x,t)\,dxdt\right)\nonumber\\
&\quad\times\left(\operatorname*{ess\,\sup}_{t\in\Gamma_{j+1}} \mint_{B_{j+1}} \tilde{w}_j^2(x,t)\,dx\right)^{\frac{p}{N}}\nonumber\\
&\leq Cr^{-\frac{p^2}{N}}\left(\frac{1}{\sigma^p(1-\sigma)^{N+s p}}+\frac{1}{(1-\sigma)^p}\right)^\frac{p}{N} \left(\frac{2^{(p+q-2) j}}{\tilde{k}^{q-2}}+\frac{2^{(N+sp+q-1)j}}{\tilde{k}^{q-p}}\right)^\frac{p}{N} \nonumber\\
&\quad\times\left[\left(\frac{1}{\sigma^p(1-\sigma)^{N+s p}}+\frac{1}{(1-\sigma)^p}\right)\left(\frac{2^{(p+q-2) j}}{\tilde{k}^{q-2}}+\frac{2^{(N+sp+q-1)j}}{\tilde{k}^{q-p}}\right)+\frac{2^{(q-p)j}}{\tilde{k}^{q-p}}\right] \nonumber\\
&\quad\times\left(\int_{\Gamma_j} \mint_{B_j} w_j^q(x,t)\,dxdt\right)^{1+\frac{p}{N}} \nonumber\\
&\leq\frac{C 2^{bj}}{r^{\frac{p^2}{N}}}\left(\frac{1}{\sigma^{\frac{p(N+p)}{N}}(1-\sigma)^{\frac{(N+p)(N+sp)}{N}}}+\frac{1}{(1-\sigma)^{\frac{p(N+p)}{N}}}\right)\nonumber\\
&\quad\times\left(\frac{1}{\tilde{k}^{q-2}}+\frac{1}{\tilde{k}^{q-p}}\right)^{1+\frac{p}{N}}\left(\int_{\Gamma_j} \mint_{B_j} w_j^q(x,t)\,dxdt\right)^{1+\frac{p}{N}}
\end{align}
with $b=(1+p/N)(N+p+q)$. Combining \eqref{4.9}, \eqref{4.10} and \eqref{4.12}, we get the desired result.
\end{proof}

\begin{lemma}
\label{lem-4-3}
Let $1<p\leq2N/(N+2)$ and $m>\max\left\{2,\frac{N(2-p)}{p}\right\}$. Let $(x_0,t_0) \in Q_T$, $0<r<1$ and $Q_r^-=B_r(x_0) \times(t_0-r^p,t_0)$ such that $\overline{B}_{r}(x_0) \subseteq \Omega$ and $\left[t_0-r^p,t_0\right] \subseteq(0,T)$. Suppose that $u\in L_{\rm{loc}}^\infty(Q_T)$ is a weak subsolution to \eqref{1.1}. Then for any $j\in \mathbb{N}$, we have
\begin{align}
\label{4.13}
&\quad\int_{\Gamma_{j+1}}\mint_{B_{j+1}}w_{j+1}^m(x,t)\,dxdt\nonumber\\
&\leq \frac{C 2^{bj}}{r^{\frac{p^2}{N}}}\left(\frac{1}{\sigma^\frac{p(N+p)}{N}(1-\sigma)^{\frac{(N+p)(N+sp)}{N}}}+\frac{1}{(1-\sigma)^{\frac{p(N+p)}{N}}}\right)\left(\frac{1}{\tilde{k}^{m-p}}+\frac{1}{\tilde{k}^{m-2}}\right)^{1+\frac{p}{N}}\nonumber\\
&\quad\times\left\|\tilde{w}_j\right\|_{L^\infty\left(Q_{j+1}^-\right)}^{m-p\kappa}\left(\int_{\Gamma_j} \mint_{B_j}w_j^m(x,t)\,dxdt\right)^{1+\frac{p}{N}},
\end{align}
where $b:=(1+p/N)(N+p+m)$, $\kappa:=1+2/N$ and $C>0$ only depends on $N,p,s,m$ and $\Lambda$.
\end{lemma}

\begin{proof}
Based on the assumptions, we have
\begin{align}
\label{4.14}
\int_{\Gamma_{j+1}} \mint_{B_{j+1}} w_{j+1}^m(x,t)\,dxdt& \leq \int_{\Gamma_{j+1}} \mint_{B_{j+1}} \tilde{w}_j^m(x,t)\,dxdt\nonumber \\
& \leq\left\|\tilde{w}_j\right\|_{L^\infty\left(Q_{j+1}^-\right)}^{m-p\kappa} \int_{\Gamma_{j+1}} \mint_{B_{j+1}} \tilde{w}_j^{p\kappa}(x,t)\,dxdt.
\end{align}
By utilizing Lemma \ref{lem-2-3}, Lemma \ref{lem-4-1} with $q=m$ and inequality \eqref{4.11} with $q=m$, it yields that
\begin{align}
\label{4.15}
&\quad\int_{\Gamma_{j+1}} \mint_{B_{j+1}} \tilde{w}_j^{p\kappa}(x,t)\,dxdt\nonumber\\
&\leq C
\left(r^p \int_{\Gamma_{j+1}} \mint_{B_{j+1}} |\nabla \tilde{w}_j(x,t)|^p\,dxdt
+\int_{\Gamma_{j+1}} \mint_{B_{j+1}} \tilde{w}_j^p(x,t)\,dxdt\right)\nonumber\\
&\quad\times\left(\operatorname*{ess\,\sup}_{t\in\Gamma_{j+1}} \mint_{B_{j+1}} \tilde{w}_j^2(x,t)\,dx\right)^{\frac{p}{N}}\nonumber\\
&\leq\frac{C 2^{bj}}{r^{\frac{p^2}{N}}}\left[\frac{1}{\sigma^{\frac{p(N+p)}{N}}(1-\sigma)^{\frac{(N+p)(N+sp)}{N}}}+\frac{1}{(1-\sigma)^{\frac{p(N+p)}{N}}}\right]\left(\frac{1}{\tilde{k}^{m-p}}+\frac{1}{\tilde{k}^{m-2}}\right)^{1+\frac{p}{N}}\nonumber\\
&\quad\times\left(\int_{\Gamma_j} \mint_{B_j}w_j^m(x,t)\,dxdt\right)^{1+\frac{p}{N}}
\end{align}
with $b:=(1+p/N)(N+p+m)$. Thus combining \eqref{4.14} and \eqref{4.15}, we get the desired inequality \eqref{4.13}.
\end{proof}

\begin{remark}
In Lemma \ref{lem-4-3}, the quantity $\sigma^{\frac{p(N+p)}{N}}$ can be removed, since $\sigma\in[1/2,1)$. In addition,
\begin{align*}
\max\left\{(1-\sigma)^{\frac{(N+p)(N+sp)}{N}},(1-\sigma)^{\frac{p(N+p)}{N}}\right\}\geq(1-\sigma)^{\frac{(N+p)^2}{N}}.
\end{align*}
Thus, we can get
\begin{align*}
&\quad\int_{\Gamma_{j+1}}\mint_{B_{j+1}}w_{j+1}^m(x,t)\,dxdt\nonumber\\
&\leq \frac{C 2^{bj}}{r^{\frac{p^2}{N}}}\cdot \frac{1}{(1-\sigma)^{\frac{(N+p)^2}{N}}}\left(\frac{1}{\tilde{k}^{m-p}}+\frac{1}{\tilde{k}^{m-2}}\right)^{1+\frac{p}{N}}\nonumber\\
&\quad\times\left\|\tilde{w}_j\right\|_{L^\infty\left(Q_{j+1}^-\right)}^{m-p\kappa}\left(\int_{\Gamma_j} \mint_{B_j}w_j^m(x,t)\,dxdt\right)^{1+\frac{p}{N}}.
\end{align*}
\end{remark}

Next, we introduce an analysis lemma which will be used later.

\begin{lemma}\rm{(\cite[Lemma 4.1]{D93})}
\label{lem-4-4}
Let $\{Y_j\}_{j=0}^\infty$ be a sequence of positive numbers such that
\begin{align*}
Y_0\leq K^{-\frac{1}{\delta}}b^{-\frac{1}{\delta^2}}\ \text{ and }\ Y_{j+1}\leq Kb^jY_j^{1+\delta},\quad j=0,1,2, \ldots
\end{align*}
for some constants $K$, $b>1$ and $\delta>0$. Then we have $\lim_{j\rightarrow\infty}Y_j=0$.
\end{lemma}

Finally, we end this section by proving the results of local boundedness.

\begin{proof}[\textbf{Proof of Theorem \ref{thm-1-2}}]
Let $r=R$, $\sigma=\frac{1}{2}$, then $r_j=\frac{R}{2}+2^{-j-1}R$. Fix $q=\max\{2,p\}$. We denote
\begin{align*}
Y_j=\int_{\Gamma_j} \mint_{B_j}\left(u-k_j\right)_+^q\,dxdt, \quad j=0,1,2, \ldots.
\end{align*}
Supposing $\tilde{k}\geq 1$ and recalling $r<1$, we derive from Lemma \ref{lem-4-2} that
\begin{align}
\label{4.16}
\frac{Y_{j+1}}{r^p} &\leq  \frac{C 2^{bj} Y_j^{1+\frac{q}{N\kappa}}}{r^{p\left(1+\frac{q}{N \kappa}\right)}\tilde{k}^{\frac{q}{\kappa}\left(\frac{q}{N}+\frac{2}{N}-\frac{p}{N}\right)}}+\frac{C 2^{bj} Y_j^{1+\frac{q}{N\kappa}}}{r^{p\left(1+\frac{q}{N \kappa}\right)}\tilde{k}^{\frac{q}{\kappa}\left(\frac{q}{N}+1-\frac{2}{p}\right)}}\nonumber\\
&\leq \frac{C 2^{bj}}{\tilde{k}^{q\left(1-\frac{q}{p\kappa}\right)}}\left(\frac{Y_j}{r^p}\right)^{1+\frac{q}{N\kappa}},
\end{align}
where $b:=(1+p/N)(N+p+q)$, $\kappa:=1+2/N$ and $C>0$ only depends on $N,p,s,\Lambda$. For any $j\in\mathbb{N}$, define $W_j=Y_j/r^p$. Thus we get
\begin{align*}
W_{j+1}\leq\frac{C 2^{bj}}{\tilde{k}^{q\left(1-\frac{q}{p\kappa}\right)}}W_j^{1+\frac{q}{N\kappa}},
\end{align*}
where $\tilde{k}$ is such that
\begin{align*}
\tilde{k} \geq \max \left\{\mathrm{Tail}_\infty\left(u_+;x_{0},R/2,t_0-R^p, t_0\right), C\left(\mint^{t_0}_{t_0-R^p} \mint_{B_R} u_+^q\,dxdt\right)^{\frac{p}{N(p\kappa-q)}} \vee 1\right\}
\end{align*}
with $C$ only depending on $N,p,s$ and $\Lambda$. Thus along with Lemma \ref{lem-4-4} we can deduce that $\lim_{j\rightarrow\infty}W_j=0$. Consequently, we get the desired result
\begin{align*}
\operatorname*{ess\,\sup}_{Q_{R/2}^-}u \leq \mathrm{Tail}_\infty\left(u_+;x_0,R/2,t_0-R^p, t_0\right)+C\left(\mint_{Q_R^-} u_+^q\,dx dt\right)^{\frac{p}{N(p\kappa-q)}} \vee 1.
\end{align*}
We now finish the proof.
\end{proof}

\begin{proof}[\textbf{Proof of Theorem \ref{thm-1-4}}]
Let $R_0=R/2$ and $R_n=R/2+\sum_{i=1}^n 2^{-i-1}R$ for $n\in\mathbb{N}^+$, and $Q_n^-=B_{R_n}(x_0) \times\left(t_0-R_n^p, t_0\right)$. Set
\begin{align*}
M_n=\operatorname{ess} \sup _{Q_n^-} u_+, \quad n=0,1,2,3, \ldots.
\end{align*}
Choosing $r=R_{n+1}$ and $\sigma r=R_{n}$, then
\begin{align*}
\sigma=\frac{1 / 2+\sum_{i=1}^{n} 2^{-i-1}}{1 / 2+\sum_{i=1}^{n+1} 2^{-i-1}} \geq \frac{1}{2}.
\end{align*}
We denote
\begin{align*}
Y_j=\int_{\Gamma_j} \mint_{B_j}\left(u-k_j\right)_+^m\,dxdt, \quad j=0,1,2, \ldots.
\end{align*}
Due to Lemma \ref{lem-4-3}, we obtain
\begin{align}
Y_{j+1} &\leq  \frac{C 2^{bj}}{R_{n+1}^{\frac{p^2}{N}}}\left\|u_+\right\|_{L^\infty\left(Q_{n+1}^-\right)}^{m-p\kappa}\left(\frac{1}{(1-\sigma)^{\frac{(N+p)^2}{N}}}+\frac{1}{(1-\sigma)^{\frac{p(N+p)}{N}}}\right)\nonumber\\
&\quad\times \left(\frac{1}{\tilde{k}^{m-p}}+\frac{1}{\tilde{k}^{m-2}}\right)^{1+\frac{p}{N}}Y_j^{1+\frac{p}{N}}\nonumber\\
&\leq \frac{C 2^{bj+dn}}{R_{n+1}^{\frac{p^2}{N}}}M_{n+1}^{m-p\kappa} \left(\frac{1}{\tilde{k}^{m-p}}+\frac{1}{\tilde{k}^{m-2}}\right)^{1+\frac{p}{N}}Y_j^{1+\frac{p}{N}}
\end{align}
with $b:=(1+p/N)(N+p+m)$ and $d=(N+p)^2/N$. For any $j\in\mathbb{N}$, define $W_j=Y_j/R_n^p$. Then we get
\begin{align*}
W_{j+1}\leq C2^{bj+dn} M_{n+1}^{m-p\kappa} \left(\frac{1}{\tilde{k}^{m-p}}+\frac{1}{\tilde{k}^{m-2}}\right)^{1+\frac{p}{N}}W_j^{1+\frac{p}{N}}.
\end{align*}
According to Lemma \ref{lem-4-4}, we can see that
$$
\lim_{j\rightarrow\infty}Y_j=0
$$
if
\begin{align}
\label{4.18}
W_{0} \leq C 2^{-\frac{dnN}{p}-\frac{bN^2}{p^2}} M_{n+1}^{-\frac{N(m-p \kappa)}{p}}\left(\frac{1}{\tilde{k}^{m-p}}+\frac{1}{\tilde{k}^{m-2}}\right)^{-\frac{p+N}{p}}.
\end{align}
To ensure the above inequality, we need choose $\tilde{k}$ properly large. Indeed,
\begin{align*}
W_0=\frac{Y_0}{R_n^p}&=R_n^{-p}\int_{t_0-R_{n+1}^p}^{t_0}\mint_{B_{R_{n+1}}}u_+^m\,dxdt\\
&\leq 2^p\mint_{t_0-R_{n+1}^p}^{t_0}\mint_{B_{R_{n+1}}}u_+^m\,dxdt.
\end{align*}
Namely,
\begin{align*}
2^p\mint_{t_0-R_{n+1}^p}^{t_0}\mint_{B_{R_{n+1}}}u_+^m\,dxdt\leq C2^{-\frac{dnN}{p}-\frac{bN^2}{p^2}} M_{n+1}^{-\frac{N(m-p\kappa)}{p}}\left(\frac{1}{\tilde{k}^{m-p}}+\frac{1}{\tilde{k}^{m-2}}\right)^{-\frac{p+N}{p}},
\end{align*}
which implies that
\begin{align*}
C 2^{\frac{dnN}{N+p}}M_{n+1}^{\frac{N(m-p\kappa)}{N+p}}\left(\mint_{t_0-R_{n+1}^p}^{t_0}\mint_{B_{R_{n+1}}}u_+^m\,dxdt\right)^{\frac{p}{N+p}}\leq\left(\frac{1}{\tilde{k}^{m-p}}+\frac{1}{\tilde{k}^{m-2}}\right)^{-1}.
\end{align*}
Therefore, we take
\begin{align*}
\tilde{k}&=C 2^{\frac{dnN}{(N+p)(m-p)}}M_{n+1}^{\frac{N(m-p\kappa)}{(N+p)(m-p)}}\left(\mint_{t_0-R_{n+1}^p}^{t_0}\mint_{B_{R_{n+1}}}u_+^m\,dxdt\right)^{\frac{p}{(N+p)(m-p)}}\\
&\quad+C 2^{\frac{dnN}{(N+p)(m-2)}}M_{n+1}^{\frac{N(m-p\kappa)}{(N+p)(m-2)}}\left(\mint_{t_0-R_{n+1}^p}^{t_0}\mint_{B_{R_{n+1}}}u_+^m\,dxdt\right)^{\frac{p}{(N+p)(m-2)}}\\
&\quad+\frac{\mathrm{Tail}_\infty\left(u_+;x_0,R_n,t_0-R_{n+1}^p, t_0\right)}{2},
\end{align*}
which makes inequality \eqref{4.18} hold true. Here the constant $C$ only depends on
 $N,p,s,m$ and $\Lambda$. Under this choice, it follows from Lemma \ref{lem-4-4} that
\begin{align}
\label{4.19}
M_n=\operatorname*{ess\,\sup}_{Q_{R_n}^-} u^+&\leq C 2^{\frac{dnN}{(N+p)(m-p)}}M_{n+1}^{\frac{N(m-p\kappa)}{(N+p)(m-p)}}\left(\mint_{t_0-R_{n+1}^p}^{t_0}\mint_{B_{R_{n+1}}}u_+^m\,dxdt\right)^{\frac{p}{(N+p)(m-p)}}\nonumber\\
&\quad+C 2^{\frac{dnN}{(N+p)(m-2)}}M_{n+1}^{\frac{N(m-p\kappa)}{(N+p)(m-2)}}\left(\mint_{t_0-R_{n+1}^p}^{t_0}\mint_{B_{R_{n+1}}}u_+^m\,dxdt\right)^{\frac{p}{(N+p)(m-2)}}\nonumber\\
&\quad+\frac{\mathrm{Tail}_\infty\left(u_+;x_0,R_n,t_0-R_{n+1}^p, t_0\right)}{2}.
\end{align}
Obverse $m>\max\{2,\frac{N(2-p)}{p}\}$ and $\kappa=1+2/N$, which indicates that
\begin{align*}
0<\frac{N(m-p\kappa)}{(N+p)(m-p)},\frac{N(m-p\kappa)}{(N+p)(m-2)}<1.
\end{align*}
Now we apply the Young's inequality with $\varepsilon$ to \eqref{4.19}, arriving at
\begin{align*}
M_{n} &\leq  \varepsilon M_{n+1}+C 2^{\frac{dnN}{(N+p)(m-p-\beta)}}\varepsilon^{-\frac{\beta}{m-p-\beta}} \left(\mint_{t_0-R^p}^{t_0}\mint_{B_R}u_+^m\,dxdt\right)^{\frac{p}{(N+p)(m-p-\beta)}}\\
&\quad+C 2^{\frac{dnN}{(N+p)(m-2-\beta)}}\varepsilon^{-\frac{\beta}{m-2-\beta}}\left(\mint_{t_0-R^p}^{t_0}\mint_{B_R}u_+^m\,dxdt\right)^{\frac{p}{(N+p)(m-2-\beta)}}\\
&\quad+\frac{\mathrm{Tail}_\infty\left(u_+;x_0,R/2,t_0-R^p,t_0\right)}{2}
\end{align*}
with $\beta=(m-p\kappa)N/(p+N)$, where we used the fact $R/2\leq R_n<R$.
Via the induction argument, we can derive
\begin{align*}
M_0&\leq \varepsilon^{n+1}M_{n+1}\\
&\quad+C\varepsilon^{-\frac{\beta}{m-p-\beta}}\left(\mint_{t_0-R^p}^{t_0}\mint_{B_R}u_+^m\,dxdt\right)^{\frac{p}{(N+p)(m-p-\beta)}}\sum_{i=0}^n \left(2{\frac{dN}{(N+p)(m-p-\beta)}}\varepsilon\right)^i\\
&\quad+C\varepsilon^{-\frac{\beta}{m-2-\beta}}\left(\mint_{t_0-R^p}^{t_0}\mint_{B_R}u_+^m\,dxdt\right)^{\frac{p}{(N+p)(m-2-\beta)}}\sum_{i=0}^n \left(2{\frac{dN}{(N+p)(m-2-\beta)}}\varepsilon\right)^i\\
&\quad+\frac{\mathrm{Tail}_\infty\left(u_+;x_0,R/2,t_0-R^p,t_0\right)}{2}\sum_{i=0}^n\varepsilon^i, \quad n=0,1,2,\ldots.
\end{align*}
It is easy to see that the sum on the right-hand side could be revised by a convergent series, provided that we take
\begin{align*}
\varepsilon=2^{-\left[{\frac{dN}{(N+p)(m-2-\beta)}}+1\right]}.
\end{align*}
Finally, letting $n\rightarrow\infty$, we deduce that
\begin{align*}
\operatorname*{ess\,\sup}_{Q_{R/2}^-} u&\leq \mathrm{Tail}_\infty\left(u_+;x_0,R/2,t_0-R^p, t_0\right)\\
&\quad+C\left(\mint_{Q_R^-} u_+^m\,dx dt\right)^{\frac{p}{(N+p)(m-2-\beta)}} \vee \left(\mint_{Q_R^-} u_+^m\,dx dt\right)^{\frac{p}{(N+p)(m-p-\beta)}},
\end{align*}
where $C>0$ depends only on $N,p,s,m$ and $\Lambda$.
\end{proof}

\section{Local H\"{o}lder continuity}
\label{sec5}

In this section we aim at establishing the H\"{o}lder continuity of weak solutions to \eqref{1.1} in the case that $p>2$, based on the local boundedness results. Before verifying this conclusion, we introduce some notations.

Let
$$
0<\alpha<\frac{p}{p-1}
$$
be a constant to be determined later. Set
\begin{align}
\label{5.1}
r_j:=\frac{\sigma^jr}{2}, \quad \omega(r_0)=\omega(r/2):=M, \quad \omega(r_j):=\left(\frac{r_j}{r_0}\right)^\alpha \omega(r_0), \quad j=0,1,2,3 \ldots
\end{align}
and
\begin{align}
\label{5.2}
M:=C\left[\mathrm{Tail}_\infty\left(u;\overline{x}_0,r/2,\overline{t}_0-r^p, \overline{t}_0+r^p\right)+\left(\mint_{Q_r}|u|^p\,dxdt\right)^{\frac{1}{2}} \vee 1\right]
\end{align}
with $C$ depending on $N,p,s,\Lambda$. Define
\begin{align*}
d_j:= \begin{cases}{\left[\varepsilon \sigma^{(j-1)\alpha} M\right]^{2-p}} & \text { if } j \geq 1, \\ 
1 & \text { if } j=0,\end{cases}
\end{align*}
where
\begin{align*}
\varepsilon=\sigma^{\frac{p}{p-1}-\alpha}.
\end{align*}
Thus it is easy to obtain
\begin{align}
\label{5.3}
\frac{1}{d_{j+1}}=\left[\varepsilon \omega(r_j)\right]^{p-2} \text { for all } j \geq 0 .
\end{align}
Denote
\begin{align*}
B_j:=B_{r_j}\left(\overline{x}_0\right) \text { and } t_j:=d_j r_j^p
\end{align*}
and
\begin{align*}
Q_j:=Q_{r_j,t_j}\left(\overline{x}_0, \overline{t}_0\right)=B_j \times\left(\overline{t}_0-t_j, \overline{t}_0+t_j\right).
\end{align*}
Hence, for $j\geq 1$, we have
\begin{align*}
4\left(\sigma^{\frac{p}{p-1}-\alpha}\right)^{2-p} r_1^p \leq r_0^p \quad\text {and} \quad 4 \sigma^{\alpha(2-p)} r_{j+1}^{p} \leq r_j^p.
\end{align*}
The above inequalities combine with the definitions of $d_j$ and $t_j$ gives that
\begin{align}
\label{5.4}
4 t_{j+1} \leq t_j \text { for all } j \geq 0.
\end{align}

Now we are going to deduce an oscillation reduction on weak solutions.

\begin{lemma}
\label{lem-5-1}
Let $p>2$ and $u$ be a local solution to \eqref{1.1}. Let $\left(\overline{x}_0,\overline{t}_0\right) \in Q_T$, $r\in(0,R]$ for some $R\in(0,1)$ and $Q_R\equiv B(\overline{x}_0) \times\left(\overline{t}_0-R^p, \overline{t}_0+R^p\right)$ such that $\overline{Q}_R \subseteq Q_T$. Then
\begin{align}
\label{5.5}
\operatorname*{ess\,osc}_{Q_j} u\leq \omega(r_j) \text { for all } j=0,1,2, \ldots.
\end{align}
\end{lemma}

\begin{proof}
We prove this Lemma by induction argument. It follows from Theorem \ref{thm-1-2} and the definition of $\omega(r_0)$ that \eqref{5.5} holds true for $j=0$. We may assume that \eqref{5.5} is valid for $i \in\{0, \ldots j\}$ with some $j \geq 0$. Then we devoted to proving it holds for $i=j+1$. It is obvious that either
\begin{align}
\label{5.6}
\frac{\left|2 Q_{j+1} \cap\left\{u \geq \operatorname{ess\,\inf}_{Q_j} u+\omega(r_j) / 2\right\}\right|}{\left|2 Q_{j+1}\right|} \geq \frac{1}{2}
\end{align}
or
\begin{align}
\label{5.7}
\frac{\left|2 Q_{j+1} \cap\left\{u \leq \operatorname{ess\,\inf}_{Q_j} u+\omega(r_j) / 2\right\}\right|}{\left|2 Q_{j+1}\right|} \geq \frac{1}{2}
\end{align}
must hold. In the case of \eqref{5.6}, we set $u_j:=u-\operatorname*{ess\,\inf}_{Q_j} u$. In the case of \eqref{5.7}, we set $u_j:=\omega(r_j)-\left(u-\operatorname*{ess\,\inf}_{Q_j} u\right)$. In all cases, we have
\begin{align}
\label{5.8}
\frac{\left|2 Q_{j+1} \cap\left\{u_j \geq \omega(r_j)/2\right\}\right|}{\left|2 Q_{j+1}\right|} \geq \frac{1}{2}
\end{align}
and
\begin{align}
\label{5.9}
0 \leq \operatorname*{ess\,\inf}_{Q_i} u_j \leq \operatorname*{ess\,\sup}_{Q_i} u_j \leq 2 \omega(r_i) \text { for } i=0, \ldots, j.
\end{align}
Now we provide an important estimate to be used later,
\begin{align}
\label{5.10}
\left[\mathrm{Tail}_\infty\left(u_j;\overline{x}_0,r_j,\overline{t}_0-t_j, \overline{t}_0+t_j\right)\right]^{p-1}\leq C \sigma^{-\alpha(p-1)}\left[\omega(r_j)\right]^{p-1} \text { for } j=0,1,2 \ldots,
\end{align}
where $C$ only depends on $N,p,s$, the difference of $p/(p-1)$ and $\alpha$. Indeed,
it is easy to see the claim is true when $j=0$. For $j\geq1$, we have
\begin{align*}
&\quad\left[\mathrm{Tail}_\infty\left(u_j; \overline{x}_0, r_j, \overline{t}_0-t_j, \overline{t}_0+t_j\right)\right]^{p-1}\nonumber\\
&=r_j^p\operatorname*{ess\,\sup}_{t \in(\overline{t}_0-t_j, \overline{t}_0+t_j)} \sum_{i=1}^j \int_{B_{i-1} \backslash B_i} \frac{\left|u_j(x, t)\right|^{p-1}}{\left|x-\overline{x}_0\right|^{N+sp}}\,dx\nonumber\\
&\quad+r_j^p \operatorname*{ess\,\sup}_{t \in(\overline{t}_0-t_j, \overline{t}_0+t_j)} \int_{\mathbb{R}^N \backslash B_0} \frac{\left|u_j(x, t)\right|^{p-1}}{\left|x-\overline{x}_0\right|^{N+sp}}\,dx\nonumber\\
&\leq r_j^p \sum_{i=1}^j\left(\operatorname*{ess\,\sup}_{Q_{i-1}}   u_j\right)^{p-1} \int_{\mathbb{R}^N \backslash B_i} \frac{1}{\left|x-\overline{x}_0\right|^{N+sp}}\,dx \nonumber\\
&\quad+r_j^p \operatorname*{ess\,\sup}_{t \in(\overline{t}_0-t_j, \overline{t}_0+t_j)} \int_{\mathbb{R}^N \backslash B_0} \frac{\left|u_j(x, t)\right|^{p-1}}{\left|x-\overline{x}_0\right|^{N+sp}}\,dx\nonumber\\
&\leq C \sum_{i=1}^{j}\left(\frac{r_j}{r_i}\right)^p[\omega(r_{i-1})]^{p-1},
\end{align*}
where in the last line we used \eqref{5.9} and the definition of $u_j$. Since $\sigma\leq1/4$ and $\alpha<p/(p-1)$, we estimate the right hand side as
\begin{align*}
&\quad\sum_{i=1}^{j}\left(\frac{r_j}{r_i}\right)^p[\omega(r_{i-1})]^{p-1}\\
&=\left[\omega(r_0)\right]^{p-1}\left(\frac{r_j}{r_0}\right)^{\alpha(p-1)} \sum_{i=1}^j\left(\frac{r_{i-1}}{r_i}\right)^{\alpha(p-1)}\left(\frac{r_j}{r_i}\right)^{p-\alpha(p-1)}\\
&=\left[\omega(r_j)\right]^{p-1} \sigma^{-\alpha(p-1)} \sum_{i=0}^{j-1} \sigma^{i\left( p-\alpha(p-1)\right)}\\
&\leq\left[\omega(r_j)\right]^{p-1} \frac{\sigma^{-\alpha(p-1)}}{1-\sigma^{p-\alpha(p-1)}} \\
&\leq\frac{4^{p-\alpha(p-1)}}{\left(p-\alpha(p-1)\right) \log 4} \sigma^{-\alpha(p-1)}\left[\omega(r_j)\right]^{p-1}.
\end{align*}
Thus, we have proved \eqref{5.10} with the constant $C$ depending on $N,p,s$, the difference of $p/(p-1)$ and $\alpha$.
Next, we define
\begin{align}
\label{5.11}
v:=\min \left\{\left[\log \left(\frac{\omega(r_j)/2+d}{u_j+d}\right)\right]_+, k\right\} \quad \text { for } k>0 .
\end{align}
It follows from Corollary \ref{cor-3-3} with $a\equiv\omega(r_j)/2$ and  $b\equiv\exp(k)$ that
\begin{align*}
&\quad\int_{\overline{t}_0-2 t_{j+1}}^{\overline{t}_0+2 t_{j+1}} \mint_{2 B_{j+1}}\left|v(x,t)-(v)_{2B_{j+1}}(t)\right|^p\,dxdt\\
&\leq C t_{j+1} d^{1-p}\left(\frac{r_{j+1}}{r_j}\right)^p\left[\mathrm{Tail }_\infty\left(u_j;\overline{x}_0,r_j,\overline{t}_0-4 t_{j+1}, \overline{t}_0+4 t_{j+1}\right)\right]^{p-1} \\
&\quad+C t_{j+1}+C d^{2-p} r_{j+1}^p+C t_{j+1} r_{j+1}^{p-sp}\\
&\leq C t_{j+1} d^{1-p}\left[\varepsilon \omega(r_j)\right]^{p-1}+C t_{j+1}+C d^{2-p} r_{j+1}^p,
\end{align*}
where in the last line we used the fact $4 t_{j+1} \leq t_j$ and \eqref{5.10}.
Choosing $d=\varepsilon\omega(r_j)$ in \eqref{5.11}, we get
\begin{align*}
d^{2-p}=d_{j+1}.
\end{align*}
Since $\alpha<p/(p-1)$, we can verify
\begin{align*}
d^{1-p}\leq r_{j+1}^{-p}.
\end{align*}
Thus, we arrive at
\begin{align*}
\int_{\overline{t}_0-2 t_{j+1}}^{\overline{t}_0+2 t_{j+1}} \mint_{2 B_{j+1}}\left|v(x,t)-(v)_{2B_{j+1}}(t)\right|^p\,dxdt\leq Ct_{j+1},
\end{align*}
where $C$ depends on $N,p,s,\Lambda$, the difference of $p/(p-1)$ and $\alpha$.
Following the calculation in Pages 38-39 in \cite{DZZ21}, there holds that
\begin{align}
\label{5.12}
\frac{\left|2 Q_{j+1} \cap\left\{u_j \leq 2 \varepsilon \omega(r_j)\right\}\right|}{\left|2 Q_{j+1}\right|} \leq \frac{\overline{C}}{\log \left(\frac{1}{\sigma}\right)},
\end{align}
where $\overline{C}>0$ depends on $N,p,s,\Lambda$, the difference of $p/(p-1)$ and $\alpha$.

In what follows, we will proceed by a suitable iteration to infer the desired oscillation decay over the domain $Q_{j+1}$. For any $i=$ $0,1,2, \ldots$, we define
\begin{align*}
&\varrho_i=r_{j+1}+2^{-i} r_{j+1}, \quad \tilde{\varrho}_i:=\frac{\varrho_i+\varrho_{i+1}}{2}, \\
&\theta_i:=t_{j+1}+2^{-i} t_{j+1}, \quad \tilde{\theta}_i:=\frac{\theta_i+\theta_{i+1}}{2}, \\
&Q^i:=B^i \times \Gamma_i:=B_{\varrho_i}(\overline{x}_0) \times(\overline{t}_0-\theta_{i}, \overline{t}_0+\theta_{i}), \\
&\tilde{Q}^i:=\tilde{B}^i \times \tilde{\varrho}_i:=B_{\tilde{\varrho} i}(\overline{x}_0) \times(\overline{t}_0-\tilde{\theta}_i, \overline{t}_0+\tilde{\theta}_i).
\end{align*}
Then we take the cut-off function $\psi_i \in C_0^\infty(\tilde{B}^i)$ and $\eta_i \in C_0^\infty(\tilde{\Gamma}_i)$ such that
\begin{align*}
0 \leq \psi_i \leq 1,\quad\left|\nabla \psi_i\right| \leq C 2^i r_{j+1}^{-1} \text { in } \tilde{B}^i, \quad \psi_i \equiv 1 \text { in } B^{i+1}
\end{align*}
and
\begin{align*}
0 \leq \eta_i \leq 1,\quad
\left|\partial_t \eta_i\right| \leq C 2^i t_{j+1}^{-1} \text { in } \tilde{\Gamma}_i, \quad \eta_i \equiv 1 \text { in } \Gamma_{i+1}.
\end{align*}
Define
\begin{align*}
k_i:=(1+2^{-i}) \varepsilon \omega(r_j), \quad v_i:=(k_i-u_j)_+.
\end{align*}
Taking $\ell=\theta_i-\theta_{i+1}, \tau_1=\overline{t}_0-\theta_{i+1}$ and $\tau_2=\overline{t}_0+\theta_{i+1}$ in Lemma \ref{lem-3-1} to get
\begin{align*}
&\quad\int_{\Gamma_{i+1}} \mint_{B^{i}}|\nabla v_i(x,t)|^p\psi^p_i(x)\,dxdt+\operatorname*{ess\,\sup}_{t\in\Gamma_{i+1}}\mint_{B^i} v_i^2(x,t) \psi_i^p(x)\,dx\nonumber\\
&\quad+\int_{\Gamma_{i+1}} \int_{B^i} \mint_{B^i} \frac{\left|v_i(x,t) \psi_i(x)-v_i(y,t) \psi_i(y)\right|^p}{|x-y|^{N+sp}}\,dxdydt\\
&\leq C \int_{\Gamma_i} \mint_{B^i} |\nabla \psi_i(x)|^p v_i^p(x,t)\eta_i^2(t)\,dxdt\\
&\quad+C \int_{\Gamma_i} \int_{B^i} \mint_{B^i} \max \left\{v_i(x,t), v_i(y,t)\right\}^p\left|\psi_i(x)-\psi_i(y)\right|^p \eta_i^2(t)\,d\mu dt \\
&\quad+C\mathop{\mathrm{ess}\,\sup}_{\stackrel{t\in\Gamma_i}{x \in \mathrm{supp} \psi_i}}\int_{\mathbb{R}^N \backslash B^i} \frac{v_i^{p-1}(y,t)}{|x-y|^{N+sp}}\,dy \int_{\Gamma_i} \mint_{B^i} v_i(x, t) \psi_i^p(x) \eta_i^2(t)\,dxdt \\
&\quad+C \int_{\Gamma_i} \mint_{B^i} v_i^2(x,t) \psi_i^p(x) \eta_i(t)\left|\partial_t \eta_i(t)\right|\,dxdt\\
&=:I_1+I_2+I_3+I_4.
\end{align*}
We estimate $I_1$ as
\begin{align*}
I_1&\leq C k_i^p2^{pi}r_{j+1}^{-p}\int_{{\Gamma}_i}\mint_{B_i} \chi_{\left\{u_j \leq k_i\right\}}(x,t)\,dxdt\\
 &\leq C 2^{pi}r_{j+1}^{-p}[\varepsilon \omega(r_j)]^p\int_{{\Gamma}_i}\mint_{B_i} \chi_{\left\{u_j \leq k_i\right\}}(x,t)\,dxdt,
\end{align*}
where we used the properties of $\psi_i$. As the computations in Pages 39-40 in \cite{DZZ21}, we derive
\begin{align*}
I_2\leq C 2^{pi} r_{j+1}^{-p}[\varepsilon \omega\left(r_j\right)]^p \int_{{\Gamma}_i}\mint_{B_i} \chi_{\left\{u_j \leq k_i\right\}}(x,t)\,dxdt,
\end{align*}
\begin{align*}
I_3\leq C 2^{(N+sp)i}r_{j+1}^{-p}[\varepsilon \omega\left(r_j\right)]^p \int_{{\Gamma}_i}\mint_{B_i} \chi_{\left\{u_j \leq k_i\right\}}(x,t)\,dxdt
\end{align*}
and
\begin{align*}
I_4\leq C 2^{spi}r_{j+1}^{-p}[\varepsilon \omega\left(r_j\right)]^p \int_{{\Gamma}_i}\mint_{B_i} \chi_{\left\{u_j \leq k_i\right\}}(x,t)\,dxdt.
\end{align*}
Since $\psi_i\equiv1$ in $B^{i+1}$, we can deduce that
\begin{align}
\label{5.13}
&\quad\int_{\Gamma_{i+1}} \mint_{B^{i+1}}|\nabla v_i(x,t)|^p\,dxdt+\int_{\Gamma_{i+1}} \int_{B^{i+1}} \mint_{B^{i+1}} \frac{\left|v_i(x,t)-v_i(y,t)\right|^p}{|x-y|^{N+sp}}\,dxdydt\nonumber\\
&\quad+\operatorname*{ess\,\sup}_{t \in \Gamma_{i+1}}\mint_{B^{i+1}} v_i^2(x,t)\,dx\nonumber\\
&\leq C 2^{(N+p) i}r_{j+1}^{-p}[\varepsilon \omega(r_j)]^p  \int_{\Gamma_i} \mint_{B^i} \chi_{\left\{u_j \leq k_i\right\}}(x,t)\,dxdt.
\end{align}
From \eqref{5.3}, we get
\begin{align}
\label{5.14}
\operatorname*{ess\,\sup}_{t \in \Gamma_{i+1}} \mint_{B^{i+1}} v_i^p(x,t)\,dx & \leq k_i^{p-2} \operatorname*{ess\,\sup}_{t \in \Gamma_{i+1}} \mint_{B^{i+1}} v_i^2(x,t)\,dx\nonumber\\
& \leq C d_{j+1}^{-1} \operatorname*{ess\,\sup}_{t \in \Gamma_{i+1}} \mint_{B^{i+1}} v_i^2(x,t)\,dx.
\end{align}
Combining \eqref{5.13} with \eqref{5.14} gives that
\begin{align}
\label{5.15}
&\quad r_{j+1}^p\mint_{\Gamma_{i+1}} \mint_{B^{i+1}} |\nabla v_i(x,t)|^p\,dxdt+\operatorname*{ess\,\sup}_{t \in \Gamma_{i+1}} \mint_{B^{i+1}} v_i^p(x,t)\,dx\nonumber\\
&\leq d_{j+1}^{-1}\int_{\Gamma_{i+1}} \mint_{B^{i+1}} |\nabla v_i(x,t)|^p\,dxdt+d_{j+1}^{-1}\operatorname*{ess\,\sup}_{t \in \Gamma_{i+1}} \mint_{B^{i+1}} v_i^2(x,t)\,dx\nonumber\\
&\leq C 2^{(N+p)i} r_{j+1}^{-p} d_{j+1}^{-1}[\varepsilon \omega(r_j)]^p \int_{\Gamma_i} \mint_{B^i} \chi_{\left\{u_j \leq k_i\right\}}(x,t)\,dxdt\nonumber \\
&\leq C2^{(N+p)i}[\varepsilon \omega(r_j)]^p A_i,
\end{align}
where $A_i$ is denoted by
\begin{align*}
A_i:=\frac{\left|Q_i \cap\left\{u_j \leq k_i\right\}\right|}{\left|Q_i\right|}.
\end{align*}
In view of Lemma \ref{lem-2-4}, we can see
\begin{align}
\label{5.16}
\int_{\Gamma_{i+1}} \mint_{B^{i+1}} v_i^{p(1+\frac{p}{N})}(x,t)\,dxdt &\leq Cr_{j+1}^p \int_{\Gamma_{i+1}} \mint_{B^{i+1}} |\nabla v_i(x,t)|^p\,dxdt\nonumber\\
&\quad\times\left(\operatorname*{ess\,\sup}_{t \in \Gamma_{i+1}} \mint_{B^{i+1}} v_i^p(x,t)\,dx\right)^{\frac{p}{N}}.
\end{align}
Merging \eqref{5.15} and \eqref{5.16} leads to

\begin{align*}
A_{i+1}\left(k_i-k_{i+1}\right)^{p(1+\frac{p}{N})}  &\leq \mint_{\Gamma_{i+1}} \mint_{B^{i+1} \cap\left\{u_j\leq k_{i+1}\right\}} v_i^{p(1+\frac{p}{N})}(x,t)\,dx dt\\
&\leq  Cr_{j+1}^p \mint_{\Gamma_{i+1}} \mint_{B^{i+1}} |\nabla v_i(x,t)|^p\,dxdt\nonumber\times\left(\operatorname*{ess\,\sup}_{t \in \Gamma_{i+1}} \mint_{B^{i+1}} v_i^p(x,t)\,dx\right)^{\frac{p}{N}}\\
&\leq  C\left[2^{(N+p)i}(\varepsilon \omega(r_j))^p A_i\right]^{1+\frac{p}{N}}.
\end{align*}
We can readily get the recursive inequality
\begin{align*}
A_{i+1} \leq \tilde{C} 2^{(N+p)(1+\frac{p}{N})i} A_i^{1+\frac{p}{N}},
\end{align*}
where $\tilde{C}$ depends on $N,p,s,\Lambda$, the difference of $p/(p-1)$ and $\alpha$.
Set
\begin{align*}
\nu^*:=\tilde{C}^{-\frac{N}{p}} 2^{\frac{-N(N+p)^2}{p^2}}.
\end{align*}
Then we take
\begin{align*}
\sigma=\min \left\{\frac{1}{4},\exp \left(-\frac{\overline{C}}{\nu^*}\right)\right\}.
\end{align*}
Utilizing the definition of $A_i$, we obtain
\begin{align*}
A_0=\frac{\left|2 Q_{j+1} \cap\left\{u_j \leq 2 \varepsilon \omega(r_j)\right\}\right|}{\left|2 Q_{j+1}\right|} \leq \nu^*.
\end{align*}
It follows from Lemma \ref{lem-4-4} that
$$
\lim_{i\rightarrow\infty}A_i=0,
$$
which implies that
\begin{align*}
u_j(x,t) \geq \varepsilon \omega(r_j) \text { in } Q_{j+1}.
\end{align*}
Recalling the definition of $u_j$, we have
\begin{align}
\label{5.17}
\operatorname*{ess\,osc}_{Q_{j+1}} u \leq(1-\varepsilon) \omega(r_j)=(1-\varepsilon) \sigma^{-\alpha} \omega(r_{j+1}).
\end{align}
Now, we pick $\alpha\in(0,p/(p-1))$ such that
\begin{align*}
\sigma^\alpha\geq 1-\varepsilon=1-\sigma^{\frac{p}{p-1}-\alpha},
\end{align*}
which together with \eqref{5.17} ensures that
\begin{align*}
\operatorname*{ess\,osc}_{Q_{j+1}} u\leq \omega(r_{j+1}).
\end{align*}
Now we finish the proof.
\end{proof}

\begin{proof}[\textbf{Proof of Theorem \ref{thm-1-5}}]
Let $p>2$.  Assume that $u$ be a local weak solution of \eqref{1.1}. Let $(x_0,t_0) \in Q_T$, $R \in(0,1)$ and $Q_R \equiv B_R(x_0) \times(t_0-R^p,t_0+R^p)$ such that $\overline{Q}_R\subseteq Q_T$. Taking $r=R$ in Lemma \ref{lem-5-1} we can get
\begin{align}
\label{5.18}
\operatorname*{ess\,osc}_{Q_j} u \leq C\left(\frac{r_j}{R}\right)^{\alpha} \omega\left(\frac{R}{2}\right) \text { for all } j \in \mathbb{N}
\end{align}
with $\alpha<p/(p-1)$, $\sigma<1/4$, where $C\geq 1$ depends on $N,p,s,\Lambda,$ and
\begin{align}
\label{5.19}
\omega\left(\frac{R}{2}\right)=\mathrm{Tail}_\infty\left(u;x_0,R/2,t_0-R^p, t_0+R^p\right)+\left(\mint_{Q_R}|u|^p \,dxdt\right)^{\frac{1}{2}} \vee 1.
\end{align}
For every $\rho\in(0,R/2]$, we have $\rho \in\left(r_{j_0+1}, r_{j_0}\right]$ for
$j_0 \in \mathbb{N}$. Choosing $d=[C\omega(R/2)]^{2-p}$, it follows that $Q_{\rho, d \rho^p} \subseteq Q_{j_0}$. By applying \eqref{5.18}, we get
\begin{align*}
\operatorname*{ess\,osc}_{Q_{\rho, d\rho^p} } u \leq \operatorname*{ess\,osc}_{Q_{j_0}} u \leq C \sigma^{-\alpha}\left(\frac{r_{j_0+1}}{R}\right)^{\alpha} \omega\left(\frac{R}{2}\right) \leq C \sigma^{-\alpha}\left(\frac{\rho}{R}\right)^{\alpha} \omega\left(\frac{R}{2}\right) .
\end{align*}
 We can obtain the H\"{o}lder continuity from the above inequality along with \eqref{5.19}.
\end{proof}

\section*{Acknowledgment}
This work was supported by the National Natural Science Foundation of China (No. 12071098).

\end{document}